\documentclass[11pt]{article}
\usepackage{geometry}
\usepackage{color}
\usepackage{float}
\usepackage{textcomp}
\usepackage{amsthm}
\usepackage{amsmath}
\usepackage{amssymb}%
\usepackage{multirow}
\usepackage{changes}
\usepackage{pifont} 

\usepackage{mathtools}
\usepackage{booktabs}    
\usepackage{subfigure} 
\usepackage{caption} 
\usepackage{multirow}
\usepackage{multicol}    
\usepackage{array}       
\usepackage{graphicx}
\graphicspath{{figuras/}}
\usepackage{empheq,amsmath}
\usepackage
[ 
lined, 
boxruled, 
commentsnumbered
]
{algorithm2e}
\DecMargin{0.1cm} 

\usepackage{etoolbox}
\DeclareFontFamily{U}{matha}{\hyphenchar\font45}
\DeclareFontShape{U}{matha}{m}{n}{
<-6> matha5 <6-7> matha6 <7-8> matha7
<8-9> matha8 <9-10> matha9
<10-12> matha10 <12-> matha12
}{}
\DeclareSymbolFont{matha}{U}{matha}{m}{n}
\DeclareFontFamily{U}{mathx}{\hyphenchar\font45}
\DeclareFontShape{U}{mathx}{m}{n}{
<-6> mathx5 <6-7> mathx6 <7-8> mathx7
<8-9> mathx8 <9-10> mathx9
<10-12> mathx10 <12-> mathx12
}{}
\DeclareSymbolFont{mathx}{U}{mathx}{m}{n}
\DeclareMathDelimiter{\vvvert}{0}{matha}{"7E}{mathx}{"17}%
\DeclarePairedDelimiterX{\normiii}[1]{\vvvert}{\vvvert}{\ifblank{#1}{\:\cdot\:}{#1}}

\newtheorem{theorem}{Theorem}[section]
\newtheorem{lemma}[theorem]{Lemma}

\newtheorem{proposition}[theorem]{Proposition}
\newtheorem{remark}[theorem]{Remark}

\newcommand{\R}{\mathbb{R}}

\newcommand{\Abs}[1]{\left|#1\right|}

\newcommand{\inner}[2]{\langle{#1},{#2}\rangle}
\newcommand{\Inner}[2]{\left\langle{#1},{#2}\right\rangle}
\newcommand{\norm}[1]{\|#1\|}
\newcommand{\Norm}[1]{\left\|#1\right\|}

\newcommand{\tx}{\tilde x}

\newcommand{\tz}{\tilde z}

\newcommand{\bi}{\begin{itemize}}
\newcommand{\ei}{\end{itemize}}
\newcommand{\ba}{\begin{array}}
\newcommand{\ea}{\end{array}}

\begin{document}

\title{An inexact version of the symmetric proximal  ADMM for solving separable convex optimization}

\author{Vando A. Adona \thanks{IME, Universidade  Federal de Goias, Campus II- Caixa Postal 131, CEP 74001-970, Goi\^ania-GO, Brazil. (E-mails: {\tt vandoadona@gmail.com} and {\tt maxlng@ufg.br}).  The work of these authors was supported in part by FAPEG/GO Grant PRONEM-201710267000532, and CNPq Grants  302666/2017-6 and 408123/2018-4.} \and Max L.N. Gon\c calves \footnotemark[1] }


\maketitle

\begin{abstract}
In this paper, we propose and analyze  an  inexact version of the symmetric proximal  alternating direction method of
multipliers (ADMM) for solving  linearly constrained optimization problems.  Basically,  the   method allows its first subproblem to be solved inexactly in such way that a relative approximate criterion is satisfied. In terms of  the iteration number $k$,   we establish   global  $\mathcal{O} (1/ \sqrt{k})$  pointwise  and $\mathcal{O} (1/ {k})$ ergodic  convergence rates of the method for a  domain of the acceleration parameters, which is consistent   with the largest known  one in the exact case.  Since   the  symmetric proximal ADMM   can be seen as a class of ADMM variants,   the new algorithm  as well as its convergence rates generalize, in particular,  many others in the literature. 
Numerical experiments illustrating the practical  advantages  of the  method are reported. 
To the best of our knowledge, this work is the first one to study an   inexact version of the  symmetric proximal ADMM.
 \\[2mm]
 \end{abstract}

%
%
%
%
%
%
%
%
%
%
%
%
%
%
%
%
%
%
%
%
%
%
%
%
%

\pagestyle{plain}

\section{Introduction}\label{sec:int}
Throughout this paper, $\mathbb{R}$, $\mathbb{R}^n$ and $\mathbb{R}^{n\times p}$ denote, respectively, the set of real numbers, the set of $n$ dimensional real column vectors and the set of ${n\times p}$ real matrices.  For any vectores $x, y \in \mathbb{R}^n,$ $\inner{x}{y}$  stands for their inner product and  $\|x\|\coloneqq \sqrt{\inner{x}{y}}$ stands for the Euclidean norm of $x$. 
The  space of
symmetric positive semidefinite (resp. definite) matrices on $\mathbb{R}^{n\times n}$ is denoted by $\mathbb{S}^{n}_{+}$ (resp. $\mathbb{S}^{ n}_{++}$). Each element $Q\in \mathbb{S}^{ n}_{+}$ induces a symmetric bilinear
form $\inner{Q(\cdot)}{\cdot}$ on $\mathbb{R}^n \times \mathbb{R}^n$ and
a  seminorm $\|\cdot\|_{Q}:=\sqrt{\inner{Q(\cdot)}{\cdot}}$ on $\mathbb{R}^n$.  The trace and determinant of a matrix $P$ are denoted by $\mbox{Tr}(P)$ and $\det(P)$, respectively. We  use $I$ and ${\bf 0}$
 to stand, respectively,  for the identity matrix and the zero matrix  with proper dimension throughout the context.

Consider the separable linearly constrained optimization problem   
\begin{equation} \label{optl}
\min \{ f(x) + g(y):  A x + B y =b\},
\end{equation}
where {$f\colon\R^{n} \to (-\infty, \infty]$  
and $g\colon\R^{p} \to (-\infty, \infty]$ are proper, closed and convex functions,} $A\in\R^{m\times n}$, $B\in\R^{m\times p}$, and $b \in\R^{m}$. 
Convex optimization problems with a separable structure such as \eqref{optl}
 appear in many applications areas such as machine learning, compressive sensing and image processing.
The augmented  Lagrangian method (see, e.g., 
\cite{Ber1}) attempts to solve~\eqref{optl} directly without taking into account its particular structure. 
To overcome this drawback, a variant of the augmented Lagrangian method, namely, the alternating direction method of multipliers (ADMM),  was proposed and studied in \cite{0352.65034,0368.65053}. 
The ADMM takes full advantage of the special structure of the problem by considering each variable separably in an
alternating form  and coupling them into 
the Lagrange multiplier updating; for detailed reviews, see  \cite{Boyd:2011,glowinski1984}.

More recently, a symmetric version of the  ADMM  was proposed in  \cite{HeMaYuan2016sym} and since then it has been studied by many authors  (see, for example, \cite{Bai2018sym, Chang2019251, GaoMa2018sym, Shen2020part, Sun2017sym, Wu2019LQP, Wu2017, Wu2018708}). The method with proximal terms added (named here as   symmetric proximal ADMM) is described as follows: let  an initial point $(x_0, y_0,\gamma_0) \in  \R^{n}\times\R^{p}\times \R^{m}$, a penalty parameter $\beta>0$, two acceleration  parameters $\tau$ and $\theta$, and two proximal  matrices  $G\in \mathbb{S}^{ n}_{+}$ and $H\in \mathbb{S}^{ m}_{+}$  be given; for $k=1,2,\ldots$ do
{
\begin{subequations} \label{eq:secpro}
\begin{align}
&x_{k}\in \arg\min_{x \in \R^{n}} \left \{ f(x) - \inner{ {\gamma}_{k-1}}{Ax}
+\frac{\beta}{2} \|Ax+By_{k-1}-b \|^2 +\frac{1}{2} \|x-x_{k-1}\|_G^2\right\},  \label{eq:xprob}\\
&\gamma_{k-\frac{1}{2}} := \gamma_{k-1}- \tau\beta\left(A{x}_{k}+B y_{k-1} - b\right), \\ \label{eq:yprob}
&y_{k}\in\arg\min_{y \in \R^{p}}
  \bigg\{ g(y) - \inner{ {\gamma}_{k+\frac12}}{By} +\frac{\beta}{2} \| Ax_k+By-b \|^2 +\frac{1}{2}\|y- y_{k-1}\|_{H}^2\bigg\},\\
&\gamma_k:= \gamma_{k-1}- \theta\beta\left(A{x}_{k}+B y_{k} - b\right).   
\end{align}
\end{subequations}}
The symmetric proximal ADMM unifies several  ADMM variants. For example, it reduces to: 
\begin{itemize}
\item the standard ADMM when $G={\bf 0}$, $H={\bf 0}$, $\tau=0$ and $\theta=1$;

\item the Fortin and Glowinski acceleration version of the proximal ADMM (for short FG-P-ADMM) when $\tau=0$; see \cite{FORTIN198397,GABAY1983299,MJR2};

\item the  generalized proximal ADMM  (for short G-P-ADMM)  with the relaxation factor $\alpha:=\tau+1$ when $\theta=1$;  see \cite{Adona2018, MR1168183}.   The proof of the  latter fact can be found, for example, in   \cite[Remark 5.8]{HeMaYuan2016sym};

\item the strictly contractive Peaceman--Rachford splitting method  studied in \cite{doi:10.1137/13090849X} when $\tau=\theta \in (0,1)$.  It is worth pointing out that if $\tau=\theta=1$, $G={\bf 0}$ and  $H={\bf 0}$,   the  symmetric proximal ADMM  corresponds to the standard Peaceman-Rachford splitting method applied to the dual of~\eqref{optl}.
\end{itemize}

As has been observed by some authors (see, e.g., \cite{ Adona2018,Bai2018sym,HeMaYuan2016sym}), the use of suitable acceleration parameters $\tau$ and $\theta$ considerably improves the numerical performances of the ADMM-type algorithms. 
We also mention  that  the proximal terms  $ \|x-x_{k-1}\|_G^2/2$  and  $\|y- y_{k-1}\|_{H}^2/2$     in  subproblems   \eqref{eq:xprob} and \eqref{eq:yprob}, respectively, can make them   easier to solve or even to have closed-form solution in some applications; see, e.g.,  \cite{attouch:hal,  PADMM_Eckstein,He2015} for discussion.

In order to ensure the convergence of the symmetric ADMM in \cite{HeMaYuan2016sym}, the   parameters $\tau$ and $\theta$  were considered into the domain
\[
\mathcal{D}:= \left\{\left(\tau,\theta\right)|\; \tau\in(-1,1),\; \theta \in (0,(1+\sqrt{5})/2),  \; \tau+\theta>0,\;  |\tau|<1+\theta-\theta^2\right\}.
\]
Later, in the multi-block symmetric ADMM setting,  the authors in  \cite{Bai2018sym} have extended this convergence domain to 
\begin{equation}\label{def:reg}
\mathcal{K}:= \left\{\left(\tau,\theta\right)|  \; \tau\leq1,\; \tau+\theta>0,\;  1+\tau+\theta-\tau\theta-\tau^2-\theta^2>0\right\},
\end{equation} 
 by using appropriate proximal terms and by assuming that the matrices  associated to the respective multi-block problem 
 have full column rank.
Note that, if $\tau=0$ (resp. $\theta=1$), the convergence domains in the above regions  are equivalent  to the classical condition $\theta \in (0,(1+\sqrt{5})/2)$ (resp. $\tau\in(-1,1)$ or,  in terms of the relaxation factor $\alpha$, $\alpha\in(0,2)$) in the  FG-P-ADMM (resp. G-P-ADMM).  We refer the reader to \cite{Adona2018, MJR2} for some complexity and numerical  results of the G-P-ADMM and  FG-P-ADMM.

It  is well-known that  implementations of the ADMM in some applications may be expensive and difficult due to the necessity to
solve exactly its two subproblems.  For applications in which  one subproblem of the ADMM is significantly more challenging to solve than the other one,  being necessary, therefore, to use iterative methods to approximately solve it,  papers  \cite{adona2018partially} and \cite{adonaCOAP} proposed partially inexact versions of the FG-P-ADMM  and G-P-ADMM, respectively, using  relative error conditions.
Essentially,  the proposed schemes allow an inexact solution $\tilde{x}_k\in \mathbb{R}^n$ with residual  $u_k\in \mathbb{R}^n$ of subproblem \eqref{eq:xprob} with $G=I/\beta$, i.e.,  
\begin{equation}\label{cond:inex34}
u_k \in \partial f(\tilde x_k) - A^*\tilde{\gamma}_{k}, 
\end{equation}
such that the relative error condition 
\begin{equation}\label{cond:inex43}
 \|\tilde x_k-x_{k-1}+\beta u_k\|^2\leq\tilde \sigma\| \tilde{\gamma}_{k}-{\gamma}_{k-1}\|^2+\hat \sigma\|\tilde x_k-x_{k-1}\|^2,
\end{equation}
is satisfied, where 
\[
\tilde{\gamma}_{k}:={\gamma}_{k-1}-\beta(A\tilde x_k +B y_{k-1} - b), \quad x_k := x_{k-1}-\beta u_k,
\]
and $\tilde \sigma$ and $\hat \sigma$ are two error 
tolerance parameters.  Recall that the 
{$\varepsilon$-subdifferential} of a  convex function $h:\R^n\to \R$
is defined by
\[
\partial_{\varepsilon}h(x):=\{u\in \R^n \,:\,h(\tilde{x})\geq h(x)+\inner{u}{\tilde{x}-x}-\varepsilon,\;\;\forall \,\tilde{x}\in \R^n\}, \quad\forall\, x\in \R^n.
\]
When $\varepsilon=0$, then $\partial_0 h(x)$ 
is denoted by $\partial h(x)$
and is called the {subdifferential} of $f$ at $x$. Note that the inclusion in \eqref{cond:inex34} is based on the first-order optimality condition for \eqref{eq:xprob}.
For the inexact FG-P-ADMM  in \cite{adona2018partially},  the domain of the acceleration factor $\theta$  was
\begin{equation}\label{cond:theta}
\theta\in \left(0,\frac{1-2\tilde \sigma+\sqrt{(1-2\tilde \sigma)^2+4(1-\tilde \sigma)}}{2(1-\tilde \sigma)}\right),
\end{equation}
whereas, for  the inexact  G-P-ADMM in \cite{adonaCOAP}, the domain of the  acceleration  factor $\tau$ was 
\begin{equation}\label{eq:alpha}
\tau\in(-1,1-\tilde\sigma) \; \mbox{ (or,  in term of the relaxation factor $\alpha,$ $ \alpha\in(0,2-\tilde\sigma)$)}.
\end{equation}
If $\tilde \sigma=0$, then   \eqref{cond:theta} and \eqref{eq:alpha} reduce, respectively, to the standard conditions $\theta \in (0,(1+\sqrt{5})/2)$ and $\tau\in(-1,1)$.  
Other inexact ADMMs with relative and/or absolute error condition were proposed in \cite{jeffCOAP, Eckstein2017App, Eckstein2017Relat, NgWangYuan, Xie2017}. It is worth pointing out that, as observed in \cite{Eckstein2017App}, approximation criteria based on  relative error are more interesting  from a computational viewpoint than those based on absolute error.

Therefore, the main goal of  this work is to present   an  inexact version of the  symmetric proximal ADMM \eqref{eq:secpro}  in which, similarly to \cite{adona2018partially,adonaCOAP}, the solution of the  first subproblem can be computed in an approximate way such that a relative error condition is satisfied.  
From the theoretical viewpoint,   the global  $\mathcal{O} (1/ \sqrt{k})$  pointwise    convergence rate is shown, which  ensures, in particular, that for a given tolerance $\rho>0$,   the algorithm generates a  $\rho-$approximate solution $(\tilde x,y,\tilde{\gamma})$  with  {residual}  
$(u,v,w)$ of the Lagrangian system associated to \eqref{optl}, i.e.,
\[ u\in\partial f(\tilde x)-A^{\ast}\tilde{\gamma}, \qquad v\in \partial g(y)-B^{\ast}\tilde{\gamma},\qquad
w= A\tilde{x}+By-b, \]
and
\[ \max\{\|u\|, \|v\|, \|w\|\}\leq \rho,\]
 in at most  $\mathcal{O}(1/\rho^2)$ iterations.
The  global $\mathcal{O} (1/ {k})$ ergodic  convergence rate is also established, which implies, in particular, 
 that  a  $\rho-$approximate solution $(\tilde x^a,y^a,\tilde{\gamma}^a)$ with  {residuals}  
$(u^a,v^a,w^a)$ and $(\varepsilon^a,\zeta^a)$ of the Lagrangian system associated to \eqref{optl}, i.e., 
\[u^a\in \partial_{\varepsilon^a}f(\tilde{x}^a)- A^*\tilde{\gamma}^a,\quad v^a\in \partial_{{\zeta^{a}}}g(y^a)- B^*\tilde{\gamma}^a,  \quad w^a=A\tilde{x}^a+By^a-b,
\] 
and 
\[\max \{\|u^a\|,\|v^a\|,\|w^a\|,\varepsilon^a,{\zeta^{a}} \}\leq \rho,\]
is  obtained  in at most $\mathcal{O}(1/\rho)$ iterations by means of the ergodic sequence. 
The   analysis  of the method is established without any assumptions on $A$ and $B$. Moreover,  the new convergence domain of  $\tau$ and $\theta$    reduces to \eqref{def:reg},  except for the case $\tau=1$ and $\theta\in (-1,1)$,  in the exact setting (see  Remark~\ref{remarkalg}(a)). 
 From  the applicability viewpoint, we report numerical experiments in order to illustrate  the efficiency of the  method for solving   real-life applications. To the best of our knowledge, this work is the first one to study an   inexact version of the  symmetric proximal ADMM.

The paper is organized as follows. Section~\ref{subsec:Admm1} presents  the  inexact  symmetric  proximal  ADMM as well as its 
pointwise and ergodic convergence rates.  Section \ref{sec:Numer} is devoted to the numerical study of the proposed method.  Some concluding remarks are  given in Section~\ref{conclusion}.

\section{Inexact  symmetric proximal  ADMM}\label{subsec:Admm1}

This section describes and investigates an  inexact version of the  symmetric proximal ADMM for solving \eqref{optl}.
Essentially,  the method allows its first subproblem to be solved inexactly in such way that a relative error condition is satisfied. 
 In particular, the new algorithm  as well as its iteration-complexity results  generalize many others in the literature.
 
We begin by formally stating the inexact algorithm.

\newpage
\begin{algorithm}[h!]
\caption{\textbf{An  inexact symmetric proximal  ADMM}} \label{alg:in:sy}
\SetKwInOut{Input}{input}\SetKwInOut{Output}{output}
\BlankLine
\textbf{Step 0.} Let an initial point $(x_0,y_0,\gamma_0)\in\mathbb{R}^n\times \mathbb{R}^p\times \mathbb{R}^m$, a penalty parameter $\beta>0$, two error 
tolerance parameters $\tilde{\sigma}, \hat{\sigma} \in [0,1)$, 
and two proximal matrices $G\in \mathbb{S}^{ n}_{++}$ and $H\in\mathbb{S}^{ p}_+$  be given. Choose the acceleration parameters  $\tau$  and  $\theta$ such that 
$(\tau,\theta)\in\mathcal{R}_{\tilde{\sigma}}$
where
\begin{equation}\label{def:Reg}
\mathcal{R}_{\tilde{\sigma}}=
\left\{\left(\tau,\theta\right)\,\Bigg{|}
\begin{array}{c}
\tau\in\left(-1,1-\tilde{\sigma}\right),\qquad\tau+\theta>0, \qquad \text{and}\\\noalign{\medskip}
\left(1-\tau^{2}\right)\left(2-\tau-\theta-\tilde{\sigma}\right)-\left(1-\theta\right)^{2}\left(1-\tau-\tilde{\sigma}\right)>0
\end{array}
\right\},
\end{equation}
and set $k=1$.
\BlankLine
{\bf Step 1.} Compute $(\tilde{x}_k,u_k)\in \mathbb{R}^n\times\mathbb{R}^n$  such that 
\begin{equation}\label{cond:inex} 
u_k \in \partial f(\tilde{x}_k) - A^*\tilde{\gamma}_{k}, 
\end{equation}
and
\begin{equation}\label{cond:inex2}
\Norm{\tilde{x}_{k}-x_{k-1}+ G^{-1}u_{k}}^{2}_{G}\leq \frac{\tilde{\sigma}}{\beta}\Norm{\tilde{\gamma}_{k}-\gamma_{k-1}}^{2}+\hat{\sigma}\Norm{\tilde{x}_{k}-x_{k-1}}^{2}_{G},
\end{equation}
where
\begin{equation} \label{tilde}
\tilde{\gamma}_{k}:={\gamma}_{k-1}-\beta(A\tilde x_k +B y_{k-1} - b).
\end{equation}
\BlankLine
{\bf Step 2.} Set 
\begin{equation}\label{mult_1}
\gamma_{k-\frac{1}{2}} := \gamma_{k-1}- \tau\beta\left(A\tilde{x}_{k}+B y_{k-1} - b\right).
\end{equation} 
\BlankLine
{\bf Step 3.} Compute an optimal solution $y_{k}\in\mathbb{R}^p$ of the subproblem 
\begin{equation} \label{g_sub}
\min_{y \in \mathbb{R}^p} \left \{ g(y) - \inner{ \gamma_{k-\frac{1}{2}}}{By} +\frac{\beta}{2}\Norm{A\tilde{x}_{k}+By - b}^2+\frac{1}{2}\|y- y_{k-1}\|_{H}^2\right\}.
\end{equation}
\BlankLine
{\bf Step 4.} Set 
\begin{equation}\label{mult_2}
x_k := x_{k-1}- G^{-1}u_k, \qquad \gamma_k := \gamma_{k-\frac{1}{2}}- \theta\beta\left(A\tilde{x}_{k}+B y_{k} - b\right),
\end{equation}
and $k\leftarrow k+1$, and go to step 1. 
\BlankLine
\end{algorithm}

We now make some relevant comments of
our approach.

\begin{remark}\label{remarkalg}
\begin{itemize}
\item[(a)] Clearly,  it follows from  the definition  in \eqref{def:Reg} that if $\left(\tau,\theta\right)\in\mathcal{R}_{\tilde{\sigma}}$, then $\theta<2$,  $(1-\tau-\tilde{\sigma})>0$ and $(2-\tau-\theta-\tilde{\sigma})>0$. Moreover, the third condition in \eqref{def:Reg} can be rewritten as  
\begin{equation*}
\left(1+\tau+\theta-\tau\theta-\tau^{2}-\theta^{2}\right)\left(1-\tau\right)+\left(\tau^{2}-2\theta+\theta^{2}\right)\tilde{\sigma}>0.
\end{equation*}
If $\tilde{\sigma}=0$, then $\tau\in(-1,1)$. Hence, it follows from the above inequality that  $\mathcal{R}_{0}$ reduces  to the region $\mathcal{K}$  in~\eqref{def:reg} with $\tau\neq 1$.
The regions $\mathcal{R}_{0}$, $\mathcal{R}_{0.3}$ and $\mathcal{R}_{0.6}$ are illustrated  in Fig.~\ref{fig:reg1}, \ref{fig:reg2}, and \ref{fig:reg3}, respectively.
Note that for some suitable choice of $(\tilde{\sigma},\tau)$,  the stepsize $\theta$ can be even chosen  greater than $(1+\sqrt{5})/2\approx1.618$.

\item[(b)] If the inaccuracy  parameters $\tilde{\sigma}$ and  $\hat{\sigma}$ are zeros,  from \eqref{cond:inex2} and the  first equality in \eqref{mult_2}, we obtain $\tilde{x}_{k}={x}_{k}$ and  $u_{k}=G(x_{k-1}-{x}_{k})$. Hence,  in view of the definition of $\tilde \gamma_k$ in \eqref{tilde} and the inclusion in~\eqref{cond:inex}, it follows that  computing $x_k$ is equivalent to solve  exactly  the  subproblem  in \eqref{eq:xprob}. Therefore, we can conclude that Algorithm~\ref{alg:in:sy} recovers  its exact version. 

\item[(c)] In order to simplify  the  updated formula  of $x_k$ in \eqref{mult_2} and the relative error condition  in \eqref{cond:inex2},  a trivial choice for  the proximal matrix  $G$ would be $I/\beta$. 

\item[(d)] If  $\tau=0$, then $(\tau,\theta)\in\mathcal{R}_{\tilde{\sigma}}$  corresponds to 
\begin{equation} \label{eq:ad34}
\theta\in \left(0,\frac{1-2\tilde{\sigma}+\sqrt{(1-2\tilde{\sigma})^2+4(1-\tilde{\sigma})}}{2(1-\tilde{\sigma})}\right),
\end{equation}
and hence Algorithm~\ref{alg:in:sy} with $G=I/\beta$ reduces to the partially inexact proximal ADMM  studied in \cite{adona2018partially}. Note also that  if $\tilde{\sigma}=0$ (exact case), then  \eqref{eq:ad34}  turns out to be  the classical condition  $\theta \in (0,(1+\sqrt{5})/2)$ for the  FG-P-ADMM; see \cite{MJR2}.

\item[(e)] If  $\theta=1$, then $(\tau,\theta)\in\mathcal{R}_{\tilde{\sigma}}$  corresponds to  $\tau\in(-1,1-\tilde\sigma)$.  By setting $\alpha:=\tau+1$, it is possible to prove (see, e.g., \cite[Remark 5.8]{HeMaYuan2016sym}) that Algorithm~\ref{alg:in:sy} with $G=I/\beta$ reduces to the inexact  generalized proximal ADMM  in \cite{adonaCOAP}.  Furthermore, if $\tilde{\sigma}=0$, then  the condition on $\tau$ becomes  the standard condition $\tau\in(-1,1)$  (or,  in term of the relaxation factor $\alpha$, $\alpha\in(0,2)$)  for the G-P-ADMM; see \cite{Adona2018}.
\end{itemize}
\end{remark}


\begin{figure}[!h]
\centering
\subfigure[$\mathcal{R}_{0}$\label{fig:reg1}]
{\includegraphics[width=0.33\textwidth]{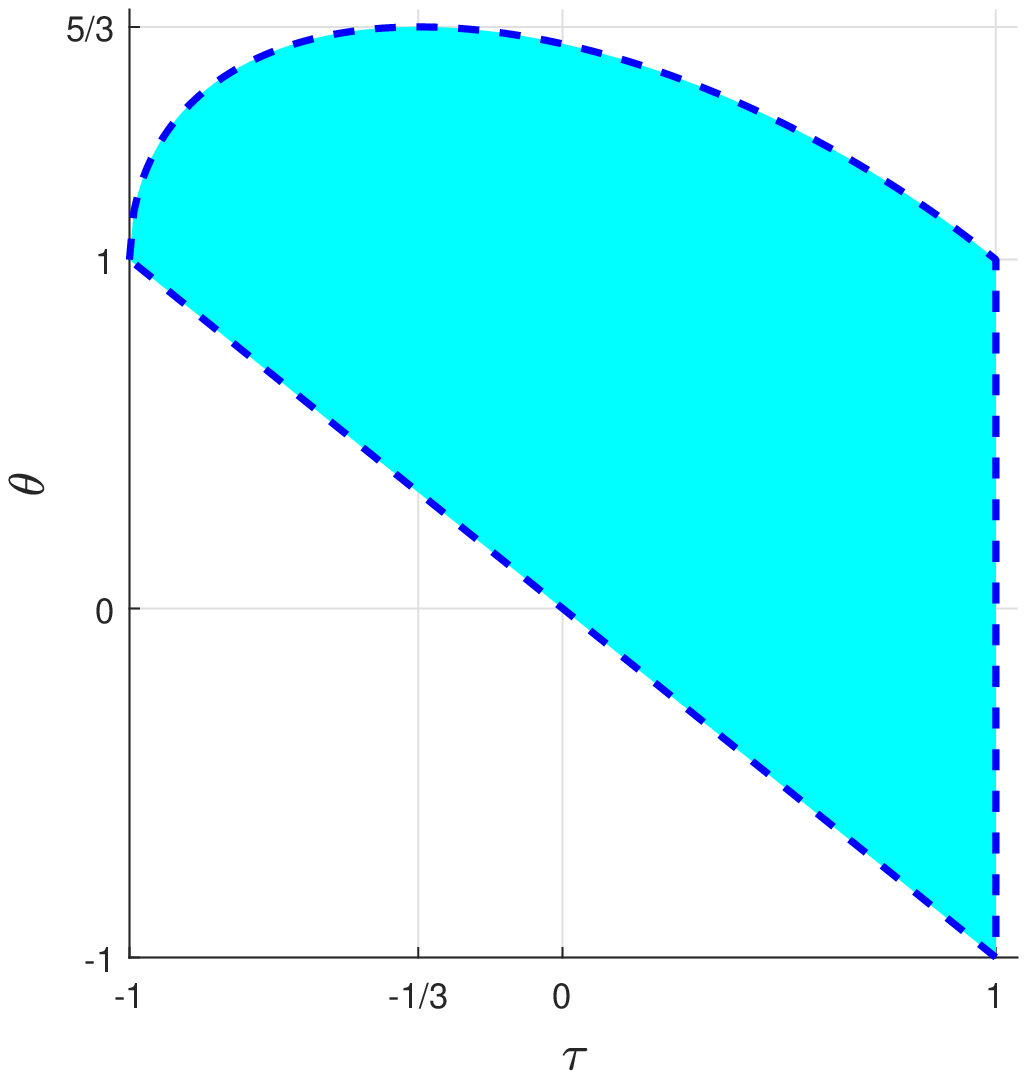}}
\hspace{-0.2cm}\subfigure[$\mathcal{R}_{0.3}$\label{fig:reg2}]
{\includegraphics[width=0.33\textwidth]{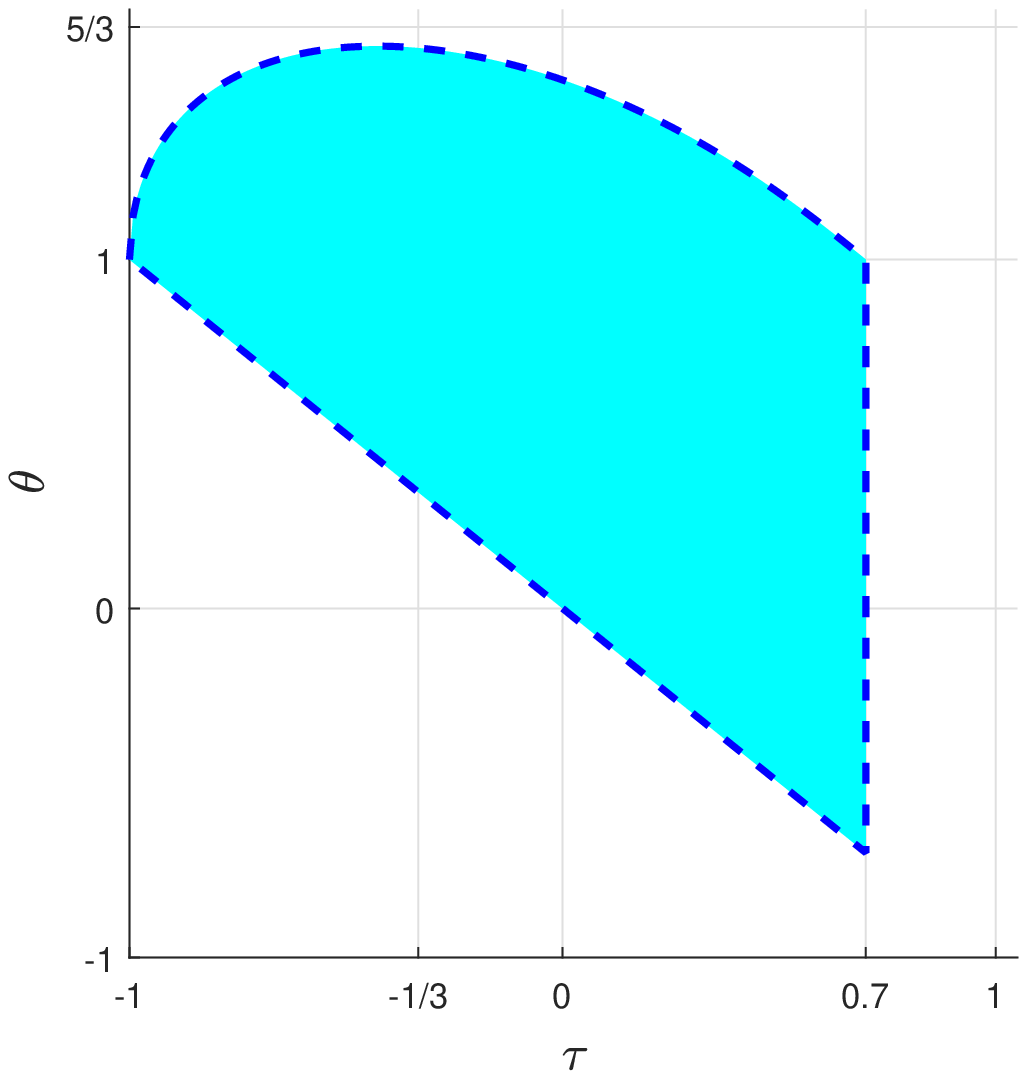}}
\hspace{-0.2cm}\subfigure[$\mathcal{R}_{0.6}$\label{fig:reg3}]
{\includegraphics[width=0.33\textwidth]{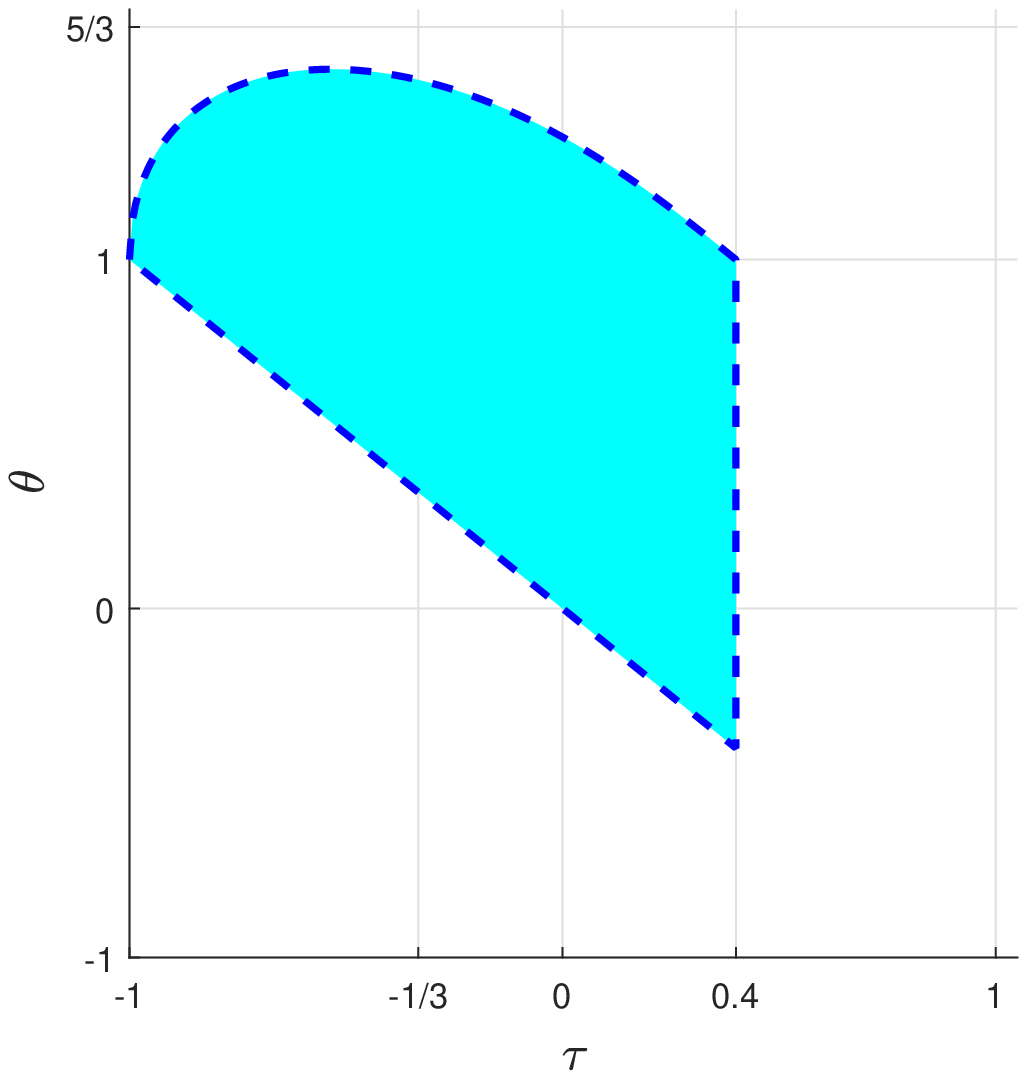}}
\caption{Some instances of $\mathcal{R}_{\tilde{\sigma}}$.}
\label{fig:testfig}
\end{figure}

Throughout the paper, we make the following standard assumption.
\\[2mm]
{\bf Assumption~1.} There exists a solution $(x^{\ast},y^{\ast},\gamma^{\ast})\in\R^n\times\R^p\times\R^m$ of the { Lagrangian system 
 \begin{equation}\label{sist:lag}
0\in\partial f( x)-A^{\ast} \gamma, \qquad 0\in \partial g( y)-B^{\ast}\gamma, \qquad 0=A x+B y-b,
\end{equation}
associated to  \eqref{optl}.}


In order to establish   pointwise and ergodic convergence rates for Algorithm~\ref{alg:in:sy}, we first  show in  Section~\ref{sec:pre2}  that the algorithm  can be seen as  an instance of a general proximal point method.     With this fact in hand,  we will be able to present  convergence rates of Algorithm~\ref{alg:in:sy} in Section~\ref{sec:bound}.    It should be mentioned  that the analysis of 
Algorithm~\ref{alg:in:sy} is much more complicated, since it involves two acceleration parameters $\tau$ and $\theta$.

\subsection{Auxiliar results}\label{sec:pre2}

Our goal in this section is to show that Algorithm~\ref{alg:in:sy} can be seen  as an instance of the hybrid proximal extragradient (HPE) framework  in \cite{MJR2} (see also \cite{Adona2018, adona2018partially}).   More specifically,  it will be proven that there exists a scalar ${\sigma}\in[\hat{\sigma},1)$ such that 
\begin{equation}\label{eq:bn45}
M\left(z_{k-1}-z_{k}\right)\in T(\tz_{k}), \qquad \Norm{\tz_{k}-z_{k}}_{M}^{2}+\eta_{k}\leq \sigma\Norm{\tz_{k}-z_{k-1}}_{M}^{2}+\eta_{k-1}, \quad \forall\, k\geq 1,
\end{equation}
where   $z_{k}:=\left(x_{k},y_{k},\gamma_{k}\right)$ and   $\tz_{k}:=\left(\tx_{k},y_{k},\tilde{\gamma}_{k}\right)$, and  the matrix $M$, the operator $T$  and the sequence $\{\eta_k\}$ will be specified  later. 
As a consequence of the latter fact, the pointwise  convergence rate  to be presented in  the next section could be derived from  \cite[Theorem 3.3]{MJR2}. However, since its proof follows easily from \eqref{eq:bn45},  we present it here  for completeness and convenience of the reader.  
On the other hand,  although   the ergodic convergence rate in the next section  is related to \cite[Theorem 3.4]{MJR2},  its proof does not follow  immediately from the latter theorem.

 The proof of \eqref{eq:bn45} is extensive and  nontrivial. 
We begin by defining  and  establishing  some properties of the matrix $M$ and the operator $T$.
\begin{proposition}\label{def:2oper}
Consider the operator $T$  and the matrix $M$  defined as 
\begin{equation}\label{def:oper}
T(x,y,\gamma)=
\left[
\begin{array}{c} 
\partial f(x)- A^{*}\gamma\\ \partial g(y)- B^{*}\gamma\\ Ax+By-b 
\end{array}
\right], \qquad
M=\left[ 
\begin{array}{ccc} 
G &\textbf{0}&\textbf{0}\\
\textbf{0}&H+\frac{\left(\tau-\tau\theta+\theta\right)\beta}{\tau+\theta}B^*B& -\frac{\tau}{\tau+\theta}B^{\ast}\\[2mm]
\textbf{0}&-\frac{\tau}{\tau+\theta}B & \frac{1}{(\tau+\theta)\beta}I
\end{array} \right].
\end{equation}
 Then, $T$ is maximal monotone and $M$ is symmetric positive semidefinite.  
\end{proposition}
\begin{proof}
Note that $T$ can be decomposed as $T=\widetilde{T}+\widehat{T}$, where
\begin{equation*}
\widetilde{T}(z):=
\left(\partial f(x), \partial g(y), -b \right) \quad\mbox{and}\quad \widehat{T}(z):= Dz, \quad\mbox{with}\qquad 
 D :=  \left[
\begin{array}{ccc} 
\textbf{0}& \textbf{0}& -A^{\ast}\\ \textbf{0}& \textbf{0}& -B^{*}\\ A& B& \textbf{0} 
\end{array}
\right].
 \end{equation*}
Thus, since $f$ and $g$ are  convex functions,  the operators $\partial f$ and $\partial g$ are maximal monotone (see \cite{Rockafellar70}) and, hence, the operator $\widetilde{T}$ is maximal monotone as well. In addition,  since $D$ is skew-symmetric, $\widehat{T}$ is also maximal monotone.  Therefore, we obtain that $T$ is maximal monotone. 

Now, it is evident that $M$ is symmetric and, using the inequality of Cauchy-Schwarz, for every $z=(x,y,\gamma)\in\mathbb{R}^n\times \mathbb{R}^p\times \mathbb{R}^m$
\begin{align} \nonumber
\Inner{Mz}{z}&= \Norm{x}_{G}^{2}+\Norm{y}_{H}^{2}+
\frac{(\tau-\tau\theta+\theta)\beta}{\tau+\theta}\Norm{By}^{2}-\frac{2\tau}{\tau+\theta}\Inner{By}{\gamma}+
\frac{1}{(\tau+\theta)\beta}\Norm{\gamma}^{2}\\\label{eq:de45}
&\geq\frac{(\tau-\tau\theta+\theta)\beta}{\tau+\theta}\Norm{By}^{2}-\frac{2\Abs{\tau}}{\tau+\theta}\Norm{By}\Norm{\gamma}+
\frac{1}{(\tau+\theta)\beta}\Norm{\gamma}^{2} 
=\Inner{Pw}{w},
\end{align}
where $w:=\left(\Norm{\gamma},\Norm{By}\right)$ and 
\[
P:=\left[ 
\begin{array}{cc} 
\frac{1}{(\tau+\theta)\beta}&-\frac{\Abs{\tau}}{\tau+\theta}\\[2mm]
  -\frac{\Abs{\tau}}{\tau+\theta} & \frac{\left(\tau-\tau\theta+\theta\right)\beta}{\tau+\theta}
\end{array} \right]. 
\]
From  step 0 of Algorithm~\ref{alg:in:sy}, we obtain 
\begin{equation*}
P_{1,1}=\frac{1}{(\tau+\theta)\beta}>0,  \quad \mbox{and}\quad
\det(P)=\frac{(1-\tau)(\tau+\theta)}{(\tau+\theta)^{2}}>0,
\end{equation*}
Therefore,  $P$ is symmetric positive definite and, hence, the statement on $M$ follows now from~\eqref{eq:de45}.
\end{proof}

We next establish a technical result.

\begin{lemma} \label{rel:tilde}
Consider the sequences $\{p_k\}$ and  $\{q_k\}$ defined by
\begin{equation}\label{def:pq}
p_{k}= B\left(y_{k}-y_{k-1}\right), \qquad  q_{k}=-\beta\left(A\tilde{x}_{k}+By_{k}-b\right), \quad \forall\, k\geq 1.
\end{equation}
 Then, for every $k\geq 1$, the following equalities hold:
\begin{align}\label{gtk:gk}
\tilde{\gamma}_{k}-\gamma_{k-1}&=\beta p_{k} + q_{k}, &  \tilde{\gamma}_{k}-\gamma_{k}=\left(1-\tau\right)\beta p_k+\left(1-\tau-\theta\right)q_{k}, \\
 \gamma_{k}-\gamma_{k-1}&=\tau\beta p_{k} +\left(\tau+\theta\right)q_{k}. & \label{gtk:gk2}
\end{align}
\end{lemma}
\begin{proof}
From the definition of $\tilde{\gamma}_{k}$ in \eqref{tilde}, we have
\begin{equation*}
\tilde{\gamma}_{k}-\gamma_{k-1}
=\beta B\left(y_{k}-y_{k-1}\right)-\beta\left(A\tx_{k}+By_{k}-b\right),
\end{equation*}
which, in view of  \eqref{def:pq}, proves the first identity in \eqref{gtk:gk}. Now, using \eqref{tilde}, \eqref{mult_1} and the definition of $\gamma_{k}$ in \eqref{mult_2} we get
\begin{align*}
\tilde{\gamma}_{k}-\gamma_{k}&= \gamma_{k-1} -\gamma_{k-\frac{1}{2}}-\beta\left(A\tx_{k}+By_{k-1}-b\right)
+\theta\beta\left(A\tx_{k}+By_{k}-b\right)\\
&=-\left(1-\tau\right)\beta\left(A\tx_{k}+By_{k-1}-b\right)+\theta\beta\left(A\tx_{k}+By_{k}-b\right)\\
&=(1-\tau)\beta B\left(y_{k}-y_{k-1}\right)-\left(1-\tau-\theta\right)\beta\left(A\tx_{k}+By_{k}-b\right).
\end{align*}
This equality, together with  \eqref{def:pq}, implies  the second identity in \eqref{gtk:gk}.  
Again using the definitions of $\gamma_{k-\frac{1}{2}}$ and  $\gamma_{k}$ in \eqref{mult_1} and \eqref{mult_2}, respectively, we obtain
\begin{align*}
\gamma_{k}-\gamma_{k-1}&=-\theta\beta\left(A\tx_{k}+By_{k}-b\right)-\tau\beta\left(A\tx_{k}+By_{k-1}-b\right)\\
&=\tau\beta B\left(y_{k}-y_{k-1}\right)-\left(\tau+\theta\right)\beta\left(A\tx_{k}+By_{k}-b\right),
\end{align*}
which, combined with \eqref{def:pq},  yields  \eqref{gtk:gk2}.
\end{proof}

We next  show that  the inclusion in  \eqref{eq:bn45} holds.

\begin{theorem} \label{pr:aux}
For every $k\geq 1$,  the following estimatives hold:
\begin{align}
G(x_{k-1}-x_{k})&\in \partial f(\tx_k)-A^*\tilde{\gamma}_k,  \label{1:incl}\\[2mm]
\left(\!H+\frac{(\tau-\tau\theta+\theta)\beta}{\tau+\theta} B^{\ast}B\!\right)\!\!(y_{k-1}\!-\!y_{k})-\frac{\tau}{\tau+\theta}B^{\ast}(\gamma_{k-1}-\gamma_{k})&\in\partial g(y_k)\!-\!B^{\ast}\tilde{\gamma}_k,\label{2:incl}\\[2mm]
-\frac{\tau}{\tau+\theta}B(y_{k-1}-y_{k})+\frac{1}{(\tau+\theta)\beta}(\gamma_{k-1}-\gamma_{k})&=A\tx_k+By_k-b.\label{3:equ}
\end{align}
As a consequence, for every  $k\geq 1$,
\[
M\left(z_{k-1}-z_{k}\right)\in T(\tz_{k}),
\]
where 
\begin{equation}\label{def:ztz}
z_{k}:=\left(x_{k},y_{k},\gamma_{k}\right)\quad\forall\, k\geq 0, \qquad \tz_{k}:=\left(\tx_{k},y_{k},\tilde{\gamma}_{k}\right)\quad\forall\,k\geq 1,
\end{equation}
and  $T$ and $M$ are as in \eqref{def:oper}.
\end{theorem}
\begin{proof}
Inclusion in  \eqref{1:incl} follows trivially from \eqref{cond:inex} and the  definition of $x_{k}$ in \eqref{mult_2}.
It follows from \eqref{mult_1} and \eqref{mult_2}  that
\begin{align*}
\gamma_k-\gamma_{k-1}&=-\theta\beta\left(A\tx_k+By_k-b \right)-\tau\beta\left(A\tx_k+By_{k-1}-b \right)\\
&=-(\tau+\theta)\beta\left(A\tx_k+By_k-b \right)+\tau\beta B\left(y_{k}-y_{k-1}\right),
\end{align*}
which is equivalent to \eqref{3:equ}. Now, from the optimality condition for \eqref{g_sub}, we have 
\begin{equation}\label{aux1:sub_y}
0\in\partial g(y_{k})- B^{\ast}\left[\gamma_{k-\frac{1}{2}}-\beta\left(A\tx_{k}+By_{k}-b\right)\right] + H\left(y_{k}-y_{k-1}\right).
\end{equation}
On the other hand, using \eqref{tilde}, we obtain
\begin{align*}
\gamma_{k-\frac{1}{2}}-\beta\left(A\tx_{k}+By_{k}-b\right)&=
\gamma_{k-\frac{1}{2}}-\beta\left(A\tx_{k}+By_{k-1}-b\right)-\beta B\left(y_{k}-y_{k-1}\right)\\
&=\tilde{\gamma}_{k}+\gamma_{k-\frac{1}{2}}-\gamma_{k-1}-\beta B\left(y_{k}-y_{k-1}\right).
\end{align*}
From the definition of $\gamma_{k}$ in \eqref{mult_2}, we find
\begin{align*}
\gamma_{k-\frac{1}{2}}-\gamma_{k-1}
&=\gamma_{k-\frac{1}{2}}-\gamma_{k}+\gamma_{k}-\gamma_{k-1}=\theta\beta\left(A\tx_{k}+By_{k}-b\right)+\gamma_{k}-\gamma_{k-1}\\
&=\theta\beta\left[\frac{\tau}{\tau+\theta}B\left(y_{k}-y_{k-1}\right)-\frac{1}{(\tau+\theta)\beta}\left(\gamma_{k}-\gamma_{k-1}\right)\right]+\gamma_{k}-\gamma_{k-1}\\
&=\frac{\tau\theta\beta}{\tau+\theta}B\left(y_{k}-y_{k-1}\right)+\frac{\tau}{\tau+\theta}\left(\gamma_{k}-\gamma_{k-1}\right),
\end{align*}
where the last equality is due to  \eqref{3:equ}.  Combining the last two equalities, we have
\begin{equation*}
\gamma_{k-\frac{1}{2}}-\beta\left(A\tx_{k}+By_{k}-b\right)=
\tilde{\gamma}_{k}-\frac{(\tau-\tau\theta+\theta)\beta}{\tau+\theta}B\left(y_{k}-y_{k-1}\right)+\frac{\tau}{\tau+\theta}\left(\gamma_{k}-\gamma_{k-1}\right),
\end{equation*}
which, combined with \eqref{aux1:sub_y},  implies \eqref{2:incl}.
\end{proof}

In the remaining part of this section, we will prove that the inequality in \eqref{eq:bn45} holds. Toward this
goal, we next establish three technical results.

\begin{lemma}\label{cor:aux2} Let $\{z_k\}$ and $\{\tilde{z}_{k}\}$ be as in   \eqref{def:ztz}. Then, 
for every $z^{\ast}\in T^{-1}(0)$, we have
\[
\Norm{z^{\ast}-z_{k}}_{M}^{2}-\Norm{z^{\ast}-z_{k-1}}_{M}^{2}
\leq \Norm{\tz_{k}-z_{k}}_{M}^{2}-\Norm{\tz_{k}-z_{k-1}}_{M}^{2}, \quad \forall\,k\geq 1.
\]
\end{lemma}
\begin{proof}
As $M\left(z_{k-1}-z_{k}\right)\in T(\tz_{k})$  (Theorem~\ref{pr:aux}), $T$ is monotone maximal (Proposition~\ref{def:2oper}) and $0 \in T(z^*)$,  we obtain $\Inner{M(z_{k-1}-z_{k})}{\tilde{z}_{k}-z^{\ast}}\geq 0$. Hence, 
\begin{align*}
\|z^{*}-z_{k}\|^{2}_{M}-\|z^{*}-z_{k-1}\|^{2}_{M}&=\|z^{*}-\tilde{z}_{k}+\tilde{z}_{k}-z_{k}\|^{2}_{M}-\|z^{*}-\tilde{z}_{k}+\tilde{z}_{k}-z_{k-1}\|^{2}_{M}\\
& = \|\tilde{z}_{k}-z_{k}\|^{2}_{M}+2\langle M(z_{k-1}- z_{k}),z^{*}-\tilde{z}_{k} \rangle-\|\tilde{z}_{k}-z_{k-1}\|^{2}_{M}\\ 
& \leq \|\tilde{z}_{k}-z_{k}\|^{2}_{M}-\|\tilde{z}_{k}-z_{k-1}\|^{2}_{M},
\end{align*}
concluding the proof.
\end{proof}


\begin{proposition} \label{pr:ang}
Define the matrix $Q$ and the scalar $\vartheta$ as 
\begin{equation}\label{eq:l90}
Q=\left[ 
\begin{array}{cc} 
\left(3-3\tau-2\tilde{\sigma}\right)\beta I &2\left(1-\tau-\tilde{\sigma}\right)I\\[2mm]
2\left(1-\tau-\tilde{\sigma}\right)I & \frac{4-\tau-\theta-2\tilde{\sigma}}{\beta}I
\end{array} \right],
\end{equation}
and 
\begin{equation}\label{def:vart} 
\vartheta=\sqrt{\left(3-3\tau-2\tilde{\sigma}\right)\left(4-\tau-\theta-2\tilde{\sigma}\right)}-2\left(1-\tau-\tilde{\sigma}\right).
\end{equation}
Then,  $Q$ is symmetric positive definite and  $\vartheta>0$. Moreover, for any $(y,\gamma)\in\R^{p}\times\R^{m}$  
\begin{equation*}\label{eq:a54}
\Norm{\left(y,\gamma\right)}_{Q}^{2}\geq - 2\vartheta\Inner{y}{\gamma}.
\end{equation*} 
\end{proposition}
\begin{proof}
Clearly $Q$ is symmetric, and is positive definite iff 
\begin{equation*}
\widehat{Q}=\left[ 
\begin{array}{cc} 
\left(3-3\tau-2\tilde{\sigma}\right)\beta &2\left(1-\tau-\tilde{\sigma}\right)\\[2mm]
2\left(1-\tau-\tilde{\sigma}\right) &\frac{4-\tau-\theta-2\tilde{\sigma}}{\beta}
\end{array} \right]
\end{equation*}
is positive definite. To show that $\widehat{Q}\in \mathbb{S}^{2}_{++}$ consider the scalars $\varrho$, $\tilde{\varrho}$, and $\hat{\varrho}$ defined by 
\begin{equation*}\label{vrho}
\varrho=\left(3-3\tau-2\tilde{\sigma}\right)\beta, \qquad 
\tilde{\varrho}=2\left(1-\tau-\tilde{\sigma}\right),\qquad\mbox{and}
\qquad \hat{\varrho}=\frac{4-\tau-\theta-2\tilde{\sigma}}{\beta}.
\end{equation*} 
Since $3-3\tau-2\tilde{\sigma}=\left(1-\tau\right)+2\left(1-\tau-\tilde{\sigma}\right)$ and  $4-\tau-\theta-2\tilde{\sigma}\!=\!\left(\tau+\theta\right)+2\left(2-\tau-\theta-\tilde{\sigma}\right)$, we obtain, from \eqref{def:Reg}, that $\varrho,\hat{\varrho}>0$. Moreover,
\begin{align*}
\varrho\hat{\varrho}-\tilde{\varrho}^{2}&=\left[\left(1-\tau\right)+2\left(1-\tau-\tilde{\sigma}\right)\right]\left(4-\tau-\theta-2\tilde{\sigma}\right)-4\left(1-\tau-\tilde{\sigma}\right)^{2}\\
&=\left(1-\tau-\tilde{\sigma}\right)\left[\left(2-\tau-\theta-\tilde{\sigma}\right)+2\left(3+\tau-\theta\right)\right]+\tilde{\sigma}\left(3-\tilde{\sigma}-\theta\right). 
\end{align*} 
From \eqref{def:Reg}, we have $(1-\tau-\tilde{\sigma})>0$,  $(2-\tau-\theta-\tilde{\sigma})>0$ and $\theta<2$. The latter inequality, together with the facts that $\tau>-1$ and $\tilde{\sigma}<1$, yields $3+\tau-\theta>0$ and $3-\tilde{\sigma}-\theta>0$. Therefore,  
$\det(\widehat{Q})>0$ and $\mbox{Tr}(\widehat{Q})>0$, and we conclude that $Q$ is positive definite. 
In addition, inequalities $\varrho\hat{\varrho}-\tilde{\varrho}^{2}>0$ and $(1-\tau-\tilde{\sigma})>0$ clearly  imply that $\vartheta>0$.

Now, for a given $\left(y,\gamma\right)\in \mathbb{R}^p\times\mathbb{R}^m$, using  \eqref{eq:l90}, \eqref{def:vart} and simple algebraic manipulations, we find
\begin{equation*}
\Norm{\left(y,\gamma\right)}_{Q}^{2}=\Norm{\sqrt{\left(3-3\tau-2\tilde{\sigma}\right)\beta}y
+\frac{\sqrt{4-\tau-\theta-2\tilde{\sigma}}}{\sqrt{\beta}}\gamma}^{2}
-2\vartheta\Inner{y}{\gamma}
\geq - 2\vartheta\Inner{y}{\gamma},
\end{equation*}
which concluded the proof of the proposition.
\end{proof}

\begin{proposition}\label{lm:coef}
Consider the functions $\varphi, \widehat{\varphi}, \widetilde{\varphi},\overline{\varphi}\colon\R\to\R$ defined by 
\begin{subequations}\label{coef}
\begin{align}
\varphi(\sigma)&= \left(1-\tau\right)\left(\sigma-1\right)+\left(1-\tau-\tilde{\sigma}\right)\left(\tau+\theta\right),\label{ph}
\\\noalign{\medskip}
\widehat{\varphi}(\sigma)&=\left(1-\tau\right)\left[\left(1+\theta\right)\sigma-1+\tau\right]-\tilde{\sigma}\left(\tau+\theta\right),\label{ph:hat}
\\\noalign{\medskip}
\widetilde{\varphi}(\sigma)&=\sigma-\left(1-\tau-\theta\right)^{2}-\tilde{\sigma}\left(\tau+\theta\right),\label{ph:til}
\\\noalign{\medskip}
\overline{\varphi}(\sigma)&=\left[\left(1+\tau\right)\widehat{\varphi}(\sigma)-2\tau\varphi(\sigma)\right]\left(1+\tau\right)\widetilde{\varphi}(\sigma)-\left(1-\theta\right)^{2}\left(\varphi(\sigma)\right)^{2}. \label{ph:bar}
\end{align}
\end{subequations}
Then, there exists a scalar ${\sigma}\in[\hat{\sigma},1)$ such that $\varphi(\sigma)\geq 0$, $\widehat{\varphi}(\sigma)\geq 0$,  $\widetilde{\varphi}(\sigma)> 0$ and  $\overline{\varphi}(\sigma)\geq 0$.
\end{proposition}
\begin{proof} Since
\begin{equation*}
\varphi(1)=\left(1-\tau-\tilde{\sigma}\right)\left(\tau+\theta\right)=\widehat{\varphi}(1),  \qquad
\widetilde{\varphi}(1)=\left(2-\tau-\theta-\tilde{\sigma}\right)\left(\tau+\theta\right),
\end{equation*}
and   
\begin{align*}
\overline{\varphi}(1)=\left(1-\tau-\tilde{\sigma}\right)\left(\tau+\theta\right)^{2}\left[\left(1-\tau^{2}\right)\left(2-\tau-\theta-\tilde{\sigma}\right)-\left(1-\theta\right)^{2}\left(1-\tau-\tilde{\sigma}\right)\right],
\end{align*}
it follows from  \eqref{def:Reg}  that  all functions defined in \eqref{coef} are positive for $\sigma=1$. Therefore, there exists  $\sigma\in [\hat{\sigma},1)$ close to $1$ such that   the statements of the proposition hold. 
\end{proof}

The following lemma provides some estimates of the sequences $\{\Norm{\tz_{k}-z_{k-1}}_{M}^{2}\}$ and $\{\Norm{\tz_{k}-z_{k}}_{M}^{2}\}$, which appear  in \eqref{eq:bn45}.

\begin{lemma}\label{lm/abc}
Let $T$, $M$, $\{p_k\}$, $\{q_k\}$, $\{z_k\}$ and $\{\tilde{z}_{k}\}$ be as in \eqref{def:oper}, \eqref{def:pq} and \eqref{def:ztz}.
Then, for every $k\geq 1$,
\begin{equation}\label{eq:f456}
\Norm{\tz_{k}-z_{k-1}}_{M}^{2}=\Norm{\tx_{k}-x_{k-1}}_{G}^{2}+\Norm{y_{k}-y_{k-1}}_{H}^{2}+ a_{k}, \qquad 
\Norm{\tz_{k}-z_{k}}_{M}^{2}=\Norm{\tx_{k}-x_{k}}_{G}^{2}
+ b_{k},
\end{equation}
where 
\begin{equation*}
a_{k}:=\frac{(1-\tau)(1+\theta)\beta}{\tau+\theta}\Norm{p_{k}}^{2}+\frac{2(1-\tau)}{\tau+\theta}\Inner{p_{k}}{q_{k}}+\frac{1}{(\tau+\theta)\beta}\Norm{q_{k}}^{2},
\end{equation*}
and
\begin{equation*}
b_{k}:=\frac{(1-\tau)^{2}\beta}{\tau+\theta}\Norm{p_{k}}^{2}+\frac{2(1-\tau)(1-\tau-\theta)}{\tau+\theta}\Inner{p_{k}}{q_{k}}+\frac{(1-\tau-\theta)^{2}}{(\tau+\theta)\beta}\Norm{q_{k}}^{2}.
\end{equation*}
\end{lemma}
\begin{proof}
It follows from \eqref{def:ztz} and the first equality in \eqref{gtk:gk}  that  
\begin{gather*}
\Norm{\tz_{k}-z_{k-1}}_{M}^{2}=
\Norm{\left(\tx_{k}-x_{k-1},y_{k}-y_{k-1},\beta p_{k}+ q_{k}\right)}_{M}^{2}.
\end{gather*}
Hence, using \eqref{def:oper} and \eqref{def:pq}, we find
\begin{equation*}
\Norm{\tz_{k}-z_{k-1}}_{M}^{2}= \Norm{\tx_{k}-x_{k-1}}_{G}^{2}+\Norm{y_{k}-y_{k-1}}_{H}^{2}+ \tilde{a}_{k}
\end{equation*}
where 
\begin{equation*}
\tilde{a}_{k}:=\frac{(\tau-\tau\theta+\theta)\beta}{\tau+\theta}\Norm{p_{k}}^{2}-\frac{2\tau}{\tau+\theta}\Inner{p_{k}}{\beta p_{k}+q_{k}}+\frac{1}{(\tau+\theta)\beta}\Norm{\beta p_{k}+q_{k}}^{2}.
\end{equation*}
By developing the right-hand side of the last expression, we have 
\begin{align*}
\tilde{a}_{k}&=
\frac{(\tau-\tau\theta+\theta-2\tau+1)\beta}{\tau+\theta}\Norm{p_{k}}^{2}-\frac{2\tau-2}{\tau+\theta}\Inner{p_{k}}{q_{k}}+\frac{1}{(\tau+\theta)\beta}\Norm{q_{k}}^{2}\\
&=\frac{(1-\tau)(1+\theta)\beta}{\tau+\theta}\Norm{p_{k}}^{2}+\frac{2(1-\tau)}{\tau+\theta}\Inner{p_{k}}{q_{k}}+\frac{1}{(\tau+\theta)\beta}\Norm{q_{k}}^{2}=a_{k}.
\end{align*}
Therefore, the first equation in \eqref{eq:f456} follows. Now, using \eqref{def:ztz}, \eqref{def:pq}, the second equality in \eqref{gtk:gk},  and the definition of $M$ in \eqref{def:oper}, we obtain
\begin{align*}
\Norm{\tz_{k}-z_{k}}_{M}^{2}&=\Norm{\left(\tx_{k}-x_{k},0,(1-\tau)\beta p_{k}+(1-\tau-\theta)q_{k}\right)}_{M}^{2}\\
&=\Norm{\tx_{k}-x_{k-1}}_{G}^{2}+\frac{1}{(\tau+\theta)\beta}\Norm{(1-\tau)\beta p_{k}+(1-\tau-\theta)q_{k}}^{2},
\end{align*}
which is equivalent to  the second equation in \eqref{eq:f456}.
\end{proof}

Before proving  the inequality in  \eqref{eq:bn45}, we establish  some other relations satisfied by the sequences generated by  
Algorithm~\ref{alg:in:sy}. To do this, we consider the following constant 
\begin{equation}\label{defd_0}
d_{0}=\inf\left\{\Norm{z^{\ast}-z_{0}}_{M}^{2}\,:\, z^{\ast}\in T^{-1}(0)\right\},
\end{equation} 
where $M$, $T$ and $z_{0}$ are as in \eqref{def:oper} and \eqref{def:ztz}. 
Note that, if $M$ is positive definite, then $d_{0}$ measures the squared distance  in the norm $\norm{\cdot}_{M}$ of the initial point $z_{0}=(x_{0},y_{0},\gamma_{0})$ to the solution set of \eqref{optl}.

\begin{lemma} \label{cond:ang}
Let  $\{p_k\}$,  $\{q_k\}$ and $d_0$ be  as in \eqref{def:pq} and \eqref{defd_0}. Then, the following  hold:\\
(a) 
\begin{equation*}
\min\left\{2\vartheta\Inner{p_{1}}{q_{1}},-\Norm{y_{1}-y_{0}}_{H}^{2}\right\}\geq-4d_{0}, 
\end{equation*}
where $\vartheta$ is as in \eqref{def:vart}.\\
(b) for every $k\geq 2$, we have
\begin{equation*}
2(1+\tau)\Inner{p_{k}}{q_{k}}\geq 2(1-\theta)\Inner{p_{k}}{q_{k-1}}- 2\tau\beta\Norm{p_{k}}^{2}+ \Norm{y_{k}-y_{k-1}}_{H}^{2}-\Norm{y_{k-1}-y_{k-2}}_{H}^{2}.
\end{equation*}
\end{lemma}
\begin{proof}
$(a)$  From  \eqref{gtk:gk2} with $k=1$, 
we have $\gamma_{1}-\gamma_{0}=\tau\beta p_{1}+(\tau+\theta)q_{1}$. Then, using \eqref{def:ztz} (with $k=1$) and the definition of $M$ in \eqref{def:oper}, we find 
\begin{align} 
\Norm{z_{1}-z_{0}}_{M}^{2}=\Norm{\left(x_{1}-x_{0},y_{1}-y_{0},\gamma_{1}-\gamma_{0}\right)}_{M}^{2}=\Norm{x_{1}-x_{0}}_{G}^{2}+\Norm{y_{1}-y_{0}}_{H}^{2} + c_{1}, \label{eq:lu8} 
\end{align}
where
\begin{align}   \nonumber
c_{1}&:=\frac{\left(\tau-\tau\theta+\theta\right)\beta}{\tau+\theta}\Norm{p_{1}}^{2}-\frac{2\tau}{\tau+\theta}\Inner{p_{1}}{\gamma_{1}-\gamma_{0}}
+\frac{1}{(\tau+\theta)\beta}\Norm{\gamma_{1}-\gamma_{0}}^{2}\\
&=\frac{\left(\tau-\tau\theta+\theta-\tau^{2}\right)\beta}{\tau+\theta}\Norm{p_{1}}^{2}+\frac{\tau+\theta}{\beta}\Norm{q_{1}}^{2}
=\left(1-\tau\right)\beta\Norm{p_{1}}^{2}+\frac{\tau+\theta}{\beta}\Norm{q_{1}}^{2}.\label{defc1}
\end{align}
Let $z^{\ast}=\left(x^{\ast},y^{\ast},\gamma^{\ast}\right)$ be an arbitrary solution of \eqref{sist:lag}, i.e., $z^{\ast}\in T^{-1}(0)$ with $T$ as in \eqref{def:oper}. Hence,  it follows from \eqref{eq:lu8}  and the fact that $\Norm{z-z'}_{M}^{2}\leq 2(\Norm{z}_{M}^{2}+\Norm{z'}_{M}^{2})$, for all $z,z'$, that 
\begin{equation}\label{ax2:ag}
c_{1}+\Norm{y_{1}-y_{0}}_{H}^{2}\leq 
\Norm{z_{1}-z_{0}}_{M}^{2}\leq 2\left(\Norm{z^{\ast}-z_{1}}_{M}^{2}+\Norm{z^{\ast}-z_{0}}_{M}^{2}\right).
\end{equation}
On the other hand, it follows from \eqref{cond:inex2} with $k=1$ and the definition of $x_{1}$ in \eqref{mult_2} that
\begin{equation*}
\Norm{\tx_{1}-x_{1}}_{G}^{2}\leq\hat{\sigma}\Norm{\tx_{1}-x_{0}}_{G}^{2}+\frac{\tilde{\sigma}}{\beta}\Vert\tilde{\gamma}_{1}-\gamma_{0}\Vert^{2}\leq
\Norm{\tx_{1}-x_{0}}_{G}^{2}+\frac{\tilde{\sigma}}{\beta}\Vert\tilde{\gamma}_{1}-\gamma_{0}\Vert^{2}, 
\end{equation*}
where, in the second inequality, we use that  $\hat{\sigma}<1$. 
Thus, using the first identity in \eqref{gtk:gk} with $k=1$, we obtain 
\begin{equation*}
\Norm{\tx_{1}-x_{0}}_{G}^{2}-\Norm{\tx_{1}-x_{1}}_{G}^{2}\geq
-\frac{\tilde{\sigma}}{\beta}\Norm{\beta p_{1}+ q_{1}}^{2}.
\end{equation*}
This inequality, together with Lemma~\ref{cor:aux2} and Lemma~\ref{lm/abc} (with $k=1$), implies that 
\begin{align*}
\Norm{z^{\ast}-z_{0}}_{M}^{2}-\Norm{z^{\ast}-z_{1}}_{M}^{2}&\geq 
\Norm{\tz_{1}-z_{0}}_{M}^{2}-\Norm{\tz_{1}-z_{1}}_{M}^{2}\\
&\geq\Norm{\tx_{1}-x_{0}}_{G}^{2}-\Norm{\tx_{1}-x_{1}}_{G}^{2}+\frac{(1-\tau)\left(\tau+\theta\right)\beta}{\tau+\theta}\Norm{p_{1}}^{2}\\
&+\frac{2(1-\tau)(\tau+\theta)}{\tau+\theta}\Inner{p_{1}}{q_{1}}+\frac{1-\left[1-\left(\tau+\theta\right)\right]^{2}}{(\tau+\theta)\beta}\Norm{q_1}^{2}\\
&\geq\left(1-\tau-\tilde{\sigma}\right)\left[\beta\Norm{p_{1}}^{2}+2\Inner{p_{1}}{q_{1}}\right]+\frac{2-\tau-\theta-\tilde{\sigma}}{\beta}\Norm{q_1}^{2}.
\end{align*}
Combining this inequality with  \eqref{ax2:ag} and using the identity in \eqref{defc1}, we find 
\begin{align*}
4\Norm{z^{\ast}-z_{0}}_{M}^{2}
&\geq\Norm{y_{1}-y_{0}}_{H}^{2}+ \left(3-3\tau-2\tilde{\sigma}\right)\beta\Norm{p_{1}}^{2} + 4\left(1-\tau-\tilde{\sigma}\right)\Inner{p_{1}}{q_{1}} + \frac{4-\tau-\theta-2\tilde{\sigma}}{\beta}\Norm{q_{1}}^{2}\\
&=\Norm{y_{1}-y_{0}}_{H}^{2} + \Norm{\left(p_{1},q_{1}\right)}_{Q}^{2},
\end{align*}
where $Q$ is as in \eqref{eq:l90}. Hence, using Proposition~\ref{pr:ang} we conclude that
\begin{equation*}
\max\left\{-\vartheta\Inner{p_{1}}{q_{1}},\Norm{y_{1}-y_{0}}_{H}^{2}\right\}
\leq 4\Norm{z^{\ast}-z_{0}}_{M}^{2}. 
\end{equation*}  
Therefore, statement $(a)$ follows from the definition of $d_{0}$ in \eqref{defd_0}.
\\[2mm]
$(b)$ It follows from the definitions of $\gamma_{k}$  and $q_{k}$ in \eqref{mult_2} and \eqref{def:pq}, respectively,  that
\[
\gamma_{k-\frac{1}{2}}-\beta\left(A\tx_{k}+By_{k}-b\right)=
\gamma_{k}-\left(1-\theta\right)\beta\left(A\tx_{k}+By_{k}-b\right)= \gamma_{k}+\left(1-\theta\right)q_{k}.
\]
Hence, since $y_{k}$ is an optimal solution of \eqref{g_sub}, we obtain, for every $k\geq 1$,
\begin{align*}
0&\in\partial g(y_{k})-B^{\ast}\left[\gamma_{k-\frac{1}{2}}-\beta\left(A\tx_{k}+By_{k}-b\right)\right] +H\left(y_{k}-y_{k-1}\right)\\ 
&=\partial g(y_{k})-B^{\ast}\left[\gamma_{k}+\left(1-\theta\right)q_{k}\right] + H\left(y_{k}-y_{k-1}\right).
\end{align*} 
Thus, the monotonicity of  $\partial g$ and the definition of $p_{k}$ in \eqref{def:pq} imply that, for eve-ry~$k\geq ~2$,
\begin{align*}
0&\leq\Inner{\gamma_{k}-\gamma_{k-1}+\left(1-\theta\right)\left(q_{k}-q_{k-1}\right)}{p_{k}}-\Norm{y_{k} -y_{k-1}}_{H}^{2}
+\Inner{H\left(y_{k-1} -y_{k-2}\right)}{y_{k} -y_{k-1}}\\
&=\Inner{\tau\beta p_{k} +\left(1+\tau\right)q_{k}-\left(1-\theta\right)q_{k-1}}{p_{k}}-\Norm{y_{k} -y_{k-1}}_{H}^{2}
+\Inner{H\left(y_{k-1} -y_{k-2}\right)}{y_{k} -y_{k-1}}\\
&\leq \tau\beta\Norm{p_{k}}^{2}+\left(1+\tau\right)\Inner{p_{k}}{q_{k}}-\left(1-\theta\right)\Inner{p_{k}}{q_{k-1}}-\frac{1}{2}\Norm{y_{k} -y_{k-1}}_{H}^{2}+\frac{1}{2}\Norm{y_{k-1} -y_{k-2}}_{H}^{2},
\end{align*}
where the second equality is due to  \eqref{gtk:gk2} and the last inequality is due to the fact that $2\Inner{Hy}{y^{\prime}}\leq \Norm{y}_{H}^{2}+\Norm{y^{\prime}}_{H}^{2}$ for all $y,y^{\prime}\in\mathbb{R}^p$.  Therefore, the desired inequality follows immediately from the last one.
\end{proof}

With the above propositions and lemmas, we  now  prove   the inequality in  \eqref{eq:bn45}.

\begin{theorem}\label{inq:hpe}
Let $\{z_{k}\}$, $\{\tz_{k}\}$ and $\{q_k\}$ be as in \eqref{def:ztz} and \eqref{def:pq} and assume that    
$\sigma\in[\hat{\sigma},1)$ is given by Proposition~\ref{lm:coef}. Consider the sequence  $\{\eta_{k}\}$ defined by  
\begin{equation}\label{def:eta}
\eta_{0}= \frac{4\left(1+\tau+\vartheta\right)\varphi\left(\sigma\right)}{\left(\tau+\theta\right)\left(1+\tau\right)\vartheta}d_{0}, \quad
\eta_{k}=\frac{\widetilde{\varphi}\left(\sigma\right)}{\left(\tau+\theta\right)\beta}\Norm{q_{k}}^{2}+\frac{\varphi\left(\sigma\right)}{\left(\tau+\theta\right)\left(1+\tau\right)}\Norm{y_{k}-y_{k-1}}_{H}^{2},\quad \forall\,k\geq 1,
\end{equation}
where  $\vartheta$, $d_{0}$, $\varphi$ and $\widetilde{\varphi}$ are as in \eqref{def:vart}, \eqref{defd_0}, \eqref{ph} and \eqref{ph:til}, respectively. 
Then, for every $k\geq 1$,  
\begin{equation}\label{pr:inq}
\Norm{\tz_{k}-z_{k}}_{M}^{2}+\eta_{k}\leq \sigma\Norm{\tz_{k}-z_{k-1}}_{M}^{2}+\eta_{k-1}, 
\end{equation}
where $M$ is as in \eqref{def:oper}.
\end{theorem}
\begin{proof}
It follows from Lemma~\ref{lm/abc} that
\begin{align} \nonumber
\sigma\Norm{\tz_{k}-z_{k-1}}_{M}^{2}-\Norm{\tz_{k}-z_{k}}_{M}^{2}&=\sigma\Norm{\tx_{k}-x_{k-1}}_{G}^{2}-\Norm{\tx_{k}-x_{k}}_{G}^{2}\\\nonumber
&+\sigma\Norm{y_{k}-y_{k-1}}_{H}^{2}
+\frac{\left(1-\tau\right)\left[\left(1+\theta\right)\sigma-\left(1-\tau\right)\right]\beta}{\tau+\theta}\Norm{p_{k}}^{2}\\
&+\frac{2\left(1-\tau\right)\left[\sigma-\left(1-\tau-\theta\right)\right]}{\tau+\theta}\Inner{p_{k}}{q_{k}}
+\frac{\sigma-\left(1-\tau-\theta\right)^2}{\left(\tau+\theta\right)\beta}\Norm{q_{k}}^{2}.\label{da:34}
\end{align}
Using the inequality in \eqref{cond:inex2}, the definition of $x_{k}$ in \eqref{mult_2} and noting that  $\sigma\geq \hat{\sigma}$,  we obtain 
\begin{align*}
\sigma\Norm{\tx_{k}-x_{k-1}}_{G}^{2}-\Norm{\tx_{k}-x_{k}}_{G}^{2}
\geq -\frac{\tilde{\sigma}}{\beta}\Norm{\tilde{\gamma}_{k}-\gamma_{k-1}}^{2}
= -\tilde{\sigma}\beta\Norm{p_{k}}^{2}-2\tilde{\sigma}\Inner{p_{k}}{q_{k}}-\frac{\tilde{\sigma}}{\beta}\Norm{q_{k}}^{2}
\end{align*}
where  the last equality  is due to the first expression in \eqref{gtk:gk}. Combining the last inequality with \eqref{da:34} and definitions in \eqref{coef}, we find
\begin{align}\label{two:case}
\sigma\Norm{\tz_{k}-z_{k-1}}_{M}^{2}&-\Norm{\tz_{k}-z_{k}}_{M}^{2}
\geq\frac{\widehat{\varphi}(\sigma)\beta}{\tau+\theta}\Norm{p_{k}}^{2} +\frac{2 \varphi\left(\sigma\right)}{\tau+\theta}\Inner{p_{k}}{q_{k}}
+\frac{\widetilde{\varphi}\left(\sigma\right)}{\left(\tau+\theta\right)\beta}\Norm{q_{k}}^{2}.
\end{align}
Let us now consider two cases: $k=1$ and $k\geq 2$.
\\[2mm]
Case 1 ($k=1$): From   \eqref{two:case} with $k=1$,    Lemma~\ref{cond:ang}$(a)$  and the fact that $\varphi(\sigma)\geq 0$, we have  
\begin{equation*}
\sigma\Norm{\tz_{1}-z_{0}}_{M}^{2}-\Norm{\tz_{1}-z_{1}}_{M}^{2}
\geq \frac{\widehat{\varphi}(\sigma)\beta}{\tau+\theta}\Norm{p_{1}}^{2}-\frac{4\varphi\left(\sigma\right)}{\left(\tau+\theta\right)\vartheta}d_{0}+ \frac{\widetilde{\varphi}\left(\sigma\right)}{\left(\tau+\theta\right)\beta}\Norm{q_{1}}^{2}. 
\end{equation*}
Hence, in view of the definitions of $\eta_{0}$ and $\eta_{1}$ in \eqref{def:eta},  we conclude that
\begin{align*}
\sigma\Norm{\tz_{1}-z_{0}}_{M}^{2}&-\Norm{\tz_{1}-z_{1}}_{M}^{2}+\eta_{0}-\eta_{1}
\geq \frac{\widehat{\varphi}(\sigma)\beta}{\tau+\theta}\Norm{p_{1}}^{2}
+\frac{\varphi\left(\sigma\right)}{\left(\tau+\theta\right)\left(1+\tau\right)}\left[4d_{0}-\Norm{y_{1}-y_{0}}_{H}^{2}\right]\geq 0,
\end{align*}
where  the last inequality is due to Lemma~\ref{cond:ang}$(a)$ and Proposition~\ref{lm:coef}. This implies that \eqref{pr:inq} holds for $k=1$.
\\[2mm]
Case 2 ($k\geq2$): It follows from Lemma~\ref{cond:ang}$(b)$ and \eqref{two:case} that
\begin{align*}
&\sigma\Norm{\tz_{k}-z_{k-1}}_{M}^{2}-\Norm{\tz_{k}-z_{k}}_{M}^{2}\geq\frac{\widehat{\varphi}(\sigma)\beta}{\tau+\theta}\Norm{p_{k}}^{2}+\frac{\widetilde{\varphi}\left(\sigma\right)}{\left(\tau+\theta\right)\beta}\Norm{q_{k}}^{2}\\
&+\frac{\varphi\left(\sigma\right)}{\left(\tau+\theta\right)\left(1+\tau\right)}
\left[2\left(1-\theta\right)\Inner{p_{k}}{q_{k-1}}- 2\tau\beta\Norm{p_{k}}^{2}+\Norm{y_{k}-y_{k-1}}_{H}^{2}-\Norm{y_{k-1}-y_{k-2}}_{H}^{2}\right],
\end{align*}
which, combined with the definition of $\eta_{k}$  in \eqref{def:eta}, yields 
\begin{align*}
\sigma\Norm{\tz_{k}-z_{k-1}}_{M}^{2}-\Norm{\tz_{k}-z_{k}}_{M}^{2} +\eta_{k-1}-\eta_{k}&\geq\frac{\left[\left(1+\tau\right)\widehat{\varphi}\left(\sigma\right)-2\tau\varphi\left(\sigma\right)\right]\beta}{\left(\tau+\theta\right)\left(1+\tau\right)}\Norm{p_{k}}^{2} \\\noalign{\medskip}
&
+\frac{2\left(1-\theta\right)\varphi\left(\sigma\right)}{\left(\tau+\theta\right)\left(1+\tau\right)}\Inner{p_{k}}{q_{k-1}}+\frac{\widetilde{\varphi}\left(\sigma\right)}{\left(\tau+\theta\right)\beta}\Norm{q_{k-1}}^{2}.
\end{align*}
For simplicity, we define constants a, b, and c by   
\begin{equation*}
a=\frac{\left[\left(1+\tau\right)\widehat{\varphi}\left(\sigma\right)-2\tau\varphi\left(\sigma\right)\right]\beta}{\left(\tau+\theta\right)\left(1+\tau\right)}, \qquad
b=\frac{\left(1-\theta\right)\varphi\left(\sigma\right)}{\left(\tau+\theta\right)\left(1+\tau\right)}, \qquad
\mbox{and}\qquad
c=\frac{\widetilde{\varphi}\left(\sigma\right)}{\left(\tau+\theta\right)\beta}.
\end{equation*} 
Hence,
\begin{equation}\label{1ax:2cs}
\sigma\Norm{\tz_{k}-z_{k-1}}_{M}^{2}-\Norm{\tz_{k}-z_{k}}_{M}^{2} +\eta_{k-1}-\eta_{k} \geq a\Norm{p_{k}}^{2}-2b\Inner{p_{k}}{q_{k-1}}+c\Norm{q_{k-1}}^{2}.
\end{equation}
Now, note that 
\begin{align*}
ac-{b}^{2}=\frac{\left[\left(1+\tau\right)\widehat{\varphi}\left(\sigma\right)-2\tau\varphi\left(\sigma\right)\right]\left(1+\tau\right)\widetilde{\varphi}\left(\sigma\right)-\left(1-\theta\right)^{2}\left(\varphi\left(\sigma\right)\right)^{2}}{\left(\tau+\theta\right)^{2}\left(1+\tau\right)^{2}}
=\frac{\overline{\varphi}\left(\sigma\right)}{\left(\tau+\theta\right)^{2}\left(1+\tau\right)^{2}},
\end{align*}
where $\overline{\varphi}$ is given in \eqref{ph:bar}. Therefore,  it follows from   Proposition~\ref{lm:coef} that $c>0$ and  $ac-{b}^{2}\geq 0$, which, combined with \eqref{1ax:2cs},   implies   that \eqref{pr:inq} also holds for $k\geq 2$.
\end{proof}

\begin{remark}
If $\tau=0$ (resp.  $\theta=1$), then  Theorems \ref{pr:aux} and \ref{inq:hpe} correspond to   Lemma~3.1 and Theorem~3.3 in \cite{adona2018partially} (resp.  \cite[Proposition~1(a)]{adonaCOAP}). 
\end{remark}

\subsection{Pointwise and ergodic convergence rates of  Algorithm~\ref{alg:in:sy}}\label{sec:bound}

In this section, we establish pointwise and ergodic convergence rates for Algorithm~\ref{alg:in:sy}.

\begin{theorem}[Pointwise convergence rate of Algorithm~\ref{alg:in:sy}] \label{th:ptw} 
Consider the sequences $\{v_{k}\}$ and $\{w_{k}\}$ defined, for every $k\geq 1$, by
\begin{align}\label{1rsd}
v_{k}&=\left(H+\frac{\left(\tau-\tau\theta+\theta\right)\beta}{\tau+\theta} B^{\ast}B\right)\left(y_{k-1}-y_{k}\right)-\frac{\tau}{\tau+\theta}B^{\ast}\left(\gamma_{k-1}-\gamma_{k}\right),\\\label{2rsd}
w_{k}&=-\frac{\tau}{\tau+\theta}B\left(y_{k-1}-y_{k}\right)+\frac{1}{\left(\tau+\theta\right)\beta}\left(\gamma_{k-1}-\gamma_k\right).
\end{align}  
Then, for every $k\geq 1$,
\begin{equation}\label{1rsl}
u_{k}\in\partial f(\tilde x_k)-A^{\ast}\tilde{\gamma}_k, \qquad v_{k}\in \partial g(y_k)-B^{\ast}\tilde{\gamma}_k,\qquad
w_{k}= A\tilde{x}_{k}+By_{k}-b,
\end{equation}
and there exists $i\leq k$ such that 
\begin{equation}\label{2rsl}
\max\left\{\Norm{u_{i}},\Norm{v_{i}},\Norm{w_{i}}\right\}\leq\sqrt{\frac{2\lambda_{M}d_0\mathcal{C}_1}{k}}
\end{equation}
where $\mathcal{C}_1:= [1+\sigma+{8\left(1+\tau+\vartheta\right)\varphi\left(\sigma\right)}/({\left(\tau+\theta\right)\left(1+\tau\right)\vartheta})]/[1-\sigma]$, 
 $\lambda_{M}$ is the largest eigenvalue of the matrix $M$  defined in \eqref{def:oper}, $ \sigma\in[\hat{\sigma},1)$ is given by  Proposition~\ref{lm:coef}    and $\vartheta$, $\varphi$ and $d_0$ are as in \eqref{def:vart}, \eqref{ph} and \eqref{defd_0}, respectively.
\end{theorem}
\begin{proof} By noting that $u_{k}=G\left(x_{k-1}-x_{k}\right)$ (see \eqref{mult_2}), 
the expressions in \eqref{1rsl} follow immediately from \eqref{1rsd}, \eqref{2rsd} and Theorem~\ref{pr:aux}.   From Theorem~\ref{pr:aux}, we  have   $\left(u_{k},v_{k},w_{k}\right)=M(z_{k-1}-z_{k})$ and hence
\begin{align*}
\Norm{\left(u_{k},v_{k},w_{k}\right)}^2&\leq\lambda_M\Norm{z_{k-1}-z_k}_{M}^{2}\leq 2\lambda_M\left[\Norm{z_{k-1}-\tilde{z}_k}_{M}^{2}+\Norm{\tz_k-z_k}_{M}^{2}\right]\nonumber\\
&\leq 2\lambda_M\left[\left(1+\sigma\right)\Norm{z_{k-1}- \tz_k}_M^2+\eta_{k-1}-\eta_{k}\right],
\end{align*}
where  the last inequality is due to \eqref{pr:inq}. On the other hand, from Lemma~\ref{cor:aux2} and \eqref{pr:inq}, we obtain
\begin{equation*}
\Norm{z^{\ast}-z_k}_{M}^{2}-\Norm{z^{\ast}-z_{k-1}}_{M}^{2}\leq \left(\sigma-1\right)\Norm{\tz_k-z_{k-1}}_{M}^{2}+\eta_{k-1}-\eta_{k},
\end{equation*} 
where  $z^{\ast}\in T^{-1}(0)$. The last two estimates and the fact that $\sigma<1$ imply that, for every $k\geq 1$
\begin{align*}
\Norm{\left(u_{k},v_{k},w_{k}\right)}^2&\leq 
2\lambda_M\left[\frac{1+\sigma}{1-\sigma}\left(\Norm{z^{\ast}-z_{k-1}}_{M}^2-\Norm{z^{\ast}-z_{k}}_{M}^{2}+\eta_{k-1}-\eta_{k}\right)+\eta_{k-1}-\eta_{k}\right]\\
&=\frac{2\lambda_M}{1-\sigma}\left[\left(1+\sigma\right)\left(\Norm{z^{\ast}-z_{k-1}}_{M}^2-\Norm{z^{\ast}-z_{k}}_{M}^2\right)+2\left(\eta_{k-1}-\eta_{k}\right)\right].
\end{align*}
By summing the above inequality from $k=1$ to  $k$, we obtain
\begin{equation}\label{eq:893}
\sum_{l=1}^{k}\Norm{\left(u_{l},v_{l},w_{l}\right)}^2\leq \frac{2\lambda_M}{1-\sigma}\left[\left(1+\sigma\right)\Norm{z^{\ast}-z_{0}}_{M}^{2}+2\eta_{0}\right],
\end{equation}
which, combined with the definitions of $d_{0}$ and $\eta_{0}$   \eqref{defd_0} and \eqref{def:eta}, respectively,  yields  
\begin{equation*}
k\left(\min_{l=1,\ldots,k}\Norm{\left(u_{l},v_{l},w_{l}\right)}^2\right)\leq \frac{2\lambda_M}{1-\sigma}\left[\left(1+\sigma\right)+ \frac{8\left(1+\tau+\vartheta\right)\varphi\left(\sigma\right)}{\left(\tau+\theta\right)\left(1+\tau\right)\vartheta} \right]d_0.
\end{equation*}
Therefore, \eqref{2rsl} follows now from the last inequality and the definition of $\mathcal{C}_1$.
\end{proof}

\begin{remark}\label{remark-pointwise}  (a)
It follows  from Theorem~\ref{th:ptw}   that, for a given tolerance $\rho>0$,   Algorithm~\ref{alg:in:sy} generates a  $\rho-$approximate solution $(\tilde x_i,y_i,\tilde{\gamma}_i)$ of \eqref{sist:lag} with  {residual}  
$(u_{i},v_{i},w_i)$, i.e.,
\[ u_{i}\in\partial f(\tilde x_i)-A^{\ast}\tilde{\gamma}_i, \qquad v_{i}\in \partial g(y_i)-B^{\ast}\tilde{\gamma}_i,\qquad
w_{i}= A\tilde{x}_{i}+By_{i}-b, \]
such that 
\[ \max\{\|u_i\|, \|v_i\|, \|w_i\|\}\leq \rho,\]
 in at most 
\begin{equation*}
\bar{k}=\left\lceil\frac{2\lambda_{M}d_{0}\mathcal{C}_1}{\rho^{2}}\right\rceil
\end{equation*}
iterations.  (b) Theorem~\ref{th:ptw} encompasses many recently pointwise convergence rates  of ADMM variants.
Namely, (i) by taking  $\tau=0$ and $G=I/\beta$, we obtain the  pointwise convergence rate of  the partially inexact proximal ADMM established in   \cite[Theorem~3.1]{adona2018partially}. Additionally,    if $\tilde{\sigma}=\hat{\sigma}=0$, the pointwise rate of the FG-P-ADMM  in  \cite[Theorem~2.1] {MJR2} is recovered.
(ii) By choosing    $\theta=1$ and $G=I/\beta$, we have the pointwise rate of the inexact proximal generalized ADMM as in \cite[Theorem~1]{adonaCOAP}.  Finally, if  $\theta=1$, $G=I/\beta$ and  $\tilde{\sigma}=\hat{\sigma}=0$,    the pointwise convergence rate of the G-P-ADMM  in \cite[Theorem~3.4]{Adona2018} is obtained.
\end{remark}

\begin{theorem}[Ergodic convergence rate of Algorithm~\ref{alg:in:sy}] \label{the_erg}       
Consider the sequences $\left\{\left(x^a_k,y^a_k,\gamma^a_k,\tilde{x}^a_k,\tilde\gamma^a_k\right)\right\}$, $\{\left(u_{k}^a,v_{k}^a,w_{k}^{a}\right)\}$, and  $\{\left(\varepsilon^a_{k},\zeta^a_{k}\right)\}$ defined, for every $k\geq 1$, by
\begin{equation}\label{erg01}
\left(x^a_k,y^a_k,\gamma^a_k,\tilde{x}^a_k,\tilde\gamma^a_k\right)=\frac{1}{k}\sum_{i=1}^k\left(x_i,y_i,\gamma_i,\tilde{x}_i,\tilde{\gamma}_i\right), \qquad
 \left(u_{k}^a,v_{k}^a,w_{k}^{a}\right)=\frac{1}{k}\sum_{i=1}^k\left(u_{i},v_{i},w_{i}\right),
\end{equation}
\begin{equation}\label{erg02}
\varepsilon^a_{k}=\frac{1}{{k}}\sum_{i=1}^k\Inner{u_{i}+A^{\ast}\tilde{\gamma}_{i}}{\tilde{x}_i-\tilde{x}_k^a},\qquad\mbox{and}\qquad\zeta^a_{k}=\frac{1}{{k}}\sum_{i=1}^k \Inner{v_{i}+B^{\ast}\tilde{\gamma}_{i}}{y_i-y_k^a},
\end{equation}
where $v_{i}$ and $w_{i}$ are as in \eqref{1rsd} and \eqref{2rsd}, respectively.
Then,  for every $k\geq 1$, there hold $\varepsilon^a_{k}\geq 0$,  $\zeta^{a}_{k}\geq 0$,  
\begin{equation}\label{1r_erg} 
u_{k}^{a}\in\partial_{\varepsilon^{a}_{k}}f\left(\tilde{x}_k^a\right)- A^{\ast}\tilde{\gamma}_k^a, \qquad v_{k}^{a}\in 
\partial_{\zeta^{a}_{k}}g\left(y_k^a\right)-B^{\ast}\tilde{\gamma}_k^a, \qquad w_{k}^{a}=A\tilde{x}_{k}^{a}+By_{k}^{a}-b,
\end{equation}
\begin{equation}\label{2r_erg}
\max\left\{\Norm{u_{k}^{a}},\Norm{v_{k}^a},\Norm{w_k^a}\right\} 
\leq \frac{2\sqrt{\lambda_{M}d_{0}\mathcal{C}_2}}{k}, \qquad \max\left\{\varepsilon^a_{k},\zeta^a_{k}\right\}\leq \frac{3d_0\mathcal{C}_3}{2k},
\end{equation}
where  $\mathcal{C}_2:=[1+{4\left(1+\tau+\vartheta\right)\varphi\left(\sigma\right)}/({\left(\tau+\theta\right)\left(1+\tau\right)\vartheta})]$ and $\mathcal{C}_3:=(3-2\sigma)\mathcal{C}_2/(1-\sigma)$,
 $\lambda_{M}$ is the largest eigenvalue of the matrix $M$  defined in \eqref{def:oper}, $ \sigma\in[\hat{\sigma},1)$ is given by  Proposition~\ref{lm:coef}   and $\vartheta$, $\varphi$ and $d_0$ are as in \eqref{def:vart}, \eqref{ph} and \eqref{defd_0}, respectively.
\end{theorem}
\begin{proof}
For every $i\geq 1$, it follows from Theorem~\ref{th:ptw} that 
\begin{equation*}
u_{i}+A^{\ast}\tilde{\gamma}_i\in\partial f(\tilde{x}_{i}),\qquad
v_{i}+B^{\ast}\tilde{\gamma}_i\in\partial g(y_i),\qquad
w_{i}= A\tilde{x}_{i}+B y_{i}-b. 
\end{equation*}
Hence, using  \eqref{erg01}, we immediately obtain the last equality in \eqref{1r_erg}. Furthermore, from the first two inclusions above, \eqref{erg01}, \eqref{erg02} and \cite[Theorem 2.1]{Goncalves2018} we conclude that, for every $k\geq 1$, $\varepsilon^a_{k}\geq 0$, $\zeta^a_{k}\geq 0$, and the inclusions in \eqref{1r_erg} hold. 
To show  \eqref{2r_erg}, we recall again that $\left(u_{i},v_{i},w_{i}\right)=M\left(z_{i-1}-z_{i}\right)$
(see the proof of Theorem~\ref{th:ptw}), which together with \eqref{erg01}, yields 
$\left(u_{k}^{a},v_{k}^a,w_{k}^a\right)=\left(1/k\right)M\left(z_{0}-z_{k}\right)$.
Then, for an arbitrary solution $z^{\ast}=\left(x^{\ast},y^{\ast},\gamma^{\ast}\right)$ of \eqref{sist:lag}, we have
\begin{align*}
\Norm{\left(u_{k}^{a},v_{k}^a,w_{k}^a\right)}^{2} 
\leq\frac{\lambda_{M}}{k^2}\Norm{z_{0}-z_{k}}^2_M 
\leq \frac{2\lambda_{M}}{k^2}\left(\Norm{z^{\ast}-z_{0}}_{M}^{2} + \Norm{z^*-z_k}^{2}_{M}\right).
\end{align*}
Combining Lemma~\ref{cor:aux2}  with \eqref{pr:inq}, we obtain, for every $k\geq 1$, that
\begin{equation}\label{es:zst}
\Norm{z^{\ast}-z_k}^{2}_{M}+\eta_{k}\leq
\Norm{z^{\ast}-z_{k-1}}^{2}_{M}+\left(\sigma-1\right)\Norm{\tz_{k}-z_{k-1}}^{2}_{M}+\eta_{k-1}
\leq\Norm{z^{\ast}-z_{k-1}}^{2}_{M}+\eta_{k-1}. 
\end{equation}
The last two expressions imply that  
\begin{equation*}\label{eq:ui56}
\Norm{\left(u_{k}^{a},v_{k}^a,w_{k}^a\right)}^{2}\leq \frac{4\lambda_{M}}{k^2}\left(\Norm{z^{\ast}-z_{0}}_{M}^{2}+\eta_{0} \right),
\end{equation*}
which, combined with the definitions of $d_{0}$ and $\eta_{0}$ in \eqref{defd_0} and \eqref{def:eta}, respectively, implies the first inequality in  \eqref{2r_erg}.
 Let us now show the second inequality in  \eqref{2r_erg}.  From definitions in \eqref{erg02}, we have
\begin{align*}
\varepsilon_{k}^a+\zeta_{k}^a
&=\dfrac{1}{k}\sum_{i=1}^k\,\Big(\Inner{u_{i}}{\tilde{x}_i-\tilde{x}_k^a}+\Inner{v_{i}}{y_i-y_k^a}+\Inner{\tilde{\gamma}_i}{A\tilde{x}_i+By_i-A\tilde{x}_k^a-By_k^a}\Big)\\ 
&=\dfrac{1}{k}\sum_{i=1}^k\,\Big(\Inner{u_{i}}{\tilde{x}_i-\tilde{x}_k^a}+\Inner{v_{i}}{y_i-y_k^a}+\Inner{\tilde{\gamma}_i}{w_{i}-w_{k}^{a}}\Big)\\
&=\dfrac{1}{k}\sum_{i=1}^k\,\Big(\Inner{u_{i}}{\tilde{x}_i-\tilde{x}_k^a}+\Inner{v_{i}}{y_i-y_k^a}+\Inner{w_i}{\tilde{\gamma}_{i}-\tilde{\gamma}_{k}^{a}}\Big),
\end{align*}
where  the second equality is due to the expressions of $w_{i}$ and $w_{k}^{a}$ in \eqref{1rsl} and \eqref{1r_erg}, respectively,   
and the third follows from the fact that  
\begin{align*}
\frac{1}{k}\sum_{i=1}^k\Inner{\tilde{\gamma}_i}{w_{i}-w_{k}^a}=\frac{1}{k}\sum_{i=1}^k\Inner{\tilde{\gamma}_i-\tilde{\gamma}_k^a}{w_{i}-w_{k}^a}= \frac{1}{k}\sum_{i=1}^k\Inner{w_{i}}{\tilde{\gamma}_i-\tilde{\gamma}_k^a}
\end{align*}
(see the definitions of $w_{k}^{a}$ and $\tilde{\gamma}_{k}^{a}$ in \eqref{erg01}).
Hence, setting $\tilde{z}_{k}^{a}=(\tilde{x}_k^a,y_k^a,\tilde{\gamma}_k^a)$, and noting that $(u_{i},v_{i},w_{i})=M\left(z_{i-1}-z_{i}\right)$ and $\tilde{z}_{i}=(\tilde{x}_i,y_i,\tilde{\gamma}_i)$, we obtain 
\begin{equation}\label{ep:ze}
\varepsilon_{k}^a+\zeta_{k}^a= \frac{1}{k}\sum_{i=1}^k\Inner{M\left(z_{i-1}-z_{i}\right)}{\tilde{z}_{i}-\tilde{z}_{k}^{a}}.
\end{equation} 
On the other hand, using  \eqref{pr:inq}, we deduce that   for all $z\in  \mathbb{R}^n\times\mathbb{R}^p\times \mathbb{R}^m$
\begin{align*}
\Norm{z-z_{i}}_{M}^{2}-\Norm{z-z_{i-1}}_{M}^{2}
&=\Norm{\tilde{z}_{i}-z_{i}}_{M}^{2}-\Norm{\tilde{z}_{i}-z_{i-1}}_{M}^{2}+2\Inner{M(z_{i-1}-z_{i})}{z-\tilde{z}_{i}}\\
&\leq (\sigma-1)\Norm{\tilde{z}_{i}-z_{i-1}}_{M}^{2} +\eta_{i-1}-\eta_{i}+2\Inner{M(z_{i-1}-z_{i})}{z-\tilde{z}_{i}},
\end{align*}
and then, since $\sigma<1$, we find
\begin{equation*}
2\sum_{i=1}^k \Inner{M(z_{i-1}-z_{i})}{\tilde{z}_{i}-z}\leq \Norm{z-z_{0}}_{M}^{2}-\Norm{z-z_{k}}_{M}^{2} +\eta_{0}-\eta_{k}\leq \Norm{z-z_{0}}_{M}^{2}+\eta_{0}.
\end{equation*}
Applying this result with $z:=\tilde{z}_{k}^{a}$ and combining with \eqref{ep:ze}, we find
\begin{align}\label{2ep:ze}
2k(\varepsilon_{k}^a+\zeta_{k}^a)\leq 
\Norm{\tilde{z}_{k}^{a}-z_{0}}_{M}^{2}+\eta_{0}
\leq \frac{1}{k}\sum_{i=1}^k \Norm{\tilde{z}_{i}-z_{0}}_{M}^{2}+\eta_{0}
\leq \max_{i=1,\ldots,k}\Norm{\tilde{z}_{i}-z_{0}}_{M}^{2}+\eta_{0},
\end{align}
where, in the second inequality, we used the convexity of  $\|\cdot\|_{M}^{2}$ and the fact that $\tilde{z}_{k}^{a}=(1/k)\sum_{i=1}^{k}\tilde{z}_{i}$. Additionally, since $\|z+z^{\prime}+z''\|_{M}^2\leq 3\left(\|z\|_{M}^2+\|z^{\prime}\|_{M}^2+\|z^{\prime\prime}\|_{M}^2\right)$, for all  $z, z^{\prime}, z^{\prime\prime}\in  \mathbb{R}^n\times\mathbb{R}^p\times \mathbb{R}^m$, we also have  
\begin{equation*}
\Norm{\tilde{z}_{i}-z_{0}}_{M}^{2}\leq 3\left[\Norm{\tilde{z}_{i}-z_{i}}_{M}^{2}+\Norm{z^{\ast}-z_{i}}_{M}^{2}+\Norm{z^{\ast}-z_{0}}_{M}^{2}\right],\qquad\forall\,i\geq 1. 
\end{equation*}
This, together with \eqref{pr:inq} and \eqref{es:zst},  implies that 
\begin{align*}
\Norm{\tilde{z}_{i}-z_{0}}_{M}^{2}
&\leq 3\left[\sigma\Norm{\tilde{z}_{i}-z_{i-1}}_{M}^{2}+\eta_{i-1}+\Norm{z^{\ast}-z_{i-1}}_{M}^{2}+\eta_{i-1}+\Norm{z^{\ast}-z_{0}}_{M}^{2}\right]\\
&\leq 3\left[\sigma\Norm{\tilde{z}_{i}-z_{i-1}}_{M}^{2}+2\left(\Norm{z^{\ast}-z_{i-1}}_{M}^{2}+\eta_{i-1}\right)+\Norm{z^{\ast}-z_{0}}_{M}^{2}\right]\\
&\leq 3\left[\sigma\Norm{\tilde{z}_{i}-z_{i-1}}_{M}^{2}+3\Norm{z^{\ast}-z_{0}}_{M}^{2}+2\eta_{0}\right],
\end{align*}
which, combined with \eqref{2ep:ze}, yields
\begin{equation*}
2k\left(\varepsilon_{k}^a+\zeta_{k}^a\right)\leq
3\left[3\left(\Norm{z^{\ast}-z_{0}}_{M}^{2}+\eta_{0}\right)+\sigma\max_{i=1,\ldots,k}\Norm{\tilde{z}_{i}-z_{i-1}}_{M}^{2}\right].
\end{equation*}
Now, from \eqref{es:zst}, it is also possible to verify that 
\begin{equation*}
\left(1-\sigma\right)\Norm{\tilde{z}_{i}-z_{i-1}}_{M}^{2}\leq\Norm{z^{\ast}-z_{i-1}}_{M}^{2}+ \eta_{i-1}\leq
\Norm{z^{\ast}-z_{0}}_{M}^{2}+ \eta_{0},
\end{equation*}
and, therefore 
\begin{equation}\label{eq:id90}
\varepsilon_{k}^a+\zeta_{k}^a\leq \frac{3(3-2\sigma)}{2(1-\sigma)k}\left(\Norm{z^{\ast}-z_{0}}_{M}^{2}+ \eta_{0}\right).
\end{equation}
Therefore,  the second inequality in \eqref{2r_erg} now follows from   the definitions of $d_{0}$ and $\eta_{0}$ in \eqref{defd_0} and \eqref{def:eta}, respectively.
\end{proof}

\begin{remark}\label{remark-ergodic} (a)
It follows from Theorem~\ref{the_erg}     that, for a given tolerance $\rho>0$, Algorithm \ref{alg:in:sy} generates a  $\rho-$approximate solution $(\tilde x^a_k,y^a_k,\tilde{\gamma}^a_k)$ of \eqref{sist:lag} with  {residuals}  
$(u^a_{k},v^a_{k},w^a_k)$ and $(\varepsilon^a_{k},\zeta^a_k)$, i.e.,
\[u_{k}^a\in \partial_{\varepsilon^a_{k}}f(\tilde{x}_k^a)- A^*\tilde{\gamma}_k^a,\qquad v_{k}^a\in \partial_{{\zeta^{a}_{k}}}g(y_k^a)- B^*\tilde{\gamma}_k^a,  \qquad w_{k}^a=A\tilde{x}_k^a+By_k^a-b,
\] 
such that 
\[\max \{\|u_{k}^a\|,\|v_{k}^a\|,\|w_{k}^a\|,\varepsilon^a_{k},{\zeta^{a}_{k}} \}\leq \rho,\]
in at most  $\bar{k}=\max\left\{k_{1},k_{2}\right\}$ iterations, where  
\begin{equation*}
k_{1}=\left\lceil\frac{2\sqrt{\lambda_{M}d_{0}\mathcal{C}_2}}{\rho}\right\rceil,\qquad\text{and}\qquad
k_{2}=\left\lceil\frac{3d_{0}\mathcal{C}_3}{2\rho}\right\rceil.
\end{equation*}
 (b)  Similarly to Theorem~\ref{th:ptw},   Theorem~\ref{the_erg} recovers, in particular,      many recently ergodic convergence rates  of ADMM variants. Namely,  (i) by taking  $\tau=0$ and $G=I/\beta$, we obtain the  ergodic convergence rate of  the partially inexact proximal ADMM established in   \cite[Theorem~3.2]{adona2018partially}. Additionally,    if $\tilde{\sigma}=\hat{\sigma}=0$, the ergodic rate of the FG-P-ADMM with $\theta \in (0,(1+\sqrt{5})/2)$  in  \cite[Theorem~2.2] {MJR2} is obtained.
(ii) By choosing    $\theta=1$ and $G=I/\beta$, we have the ergodic rate of the inexact proximal generalized ADMM as in \cite[Theorem~2]{adonaCOAP}.  Finally, if  $\theta=1$, $G=I/\beta$ and  $\tilde{\sigma}=\hat{\sigma}=0$,    the ergodic convergence rate of the G-P-ADMM with $\tau\in(-1,1)$  in \cite[Theorem~3.6]{Adona2018} is recuperated.
\end{remark}

\section{Numerical experiments}\label{sec:Numer} 

The purpose of this section is to assess the practical behavior  of  the proposed  method.  We first mention that
 the inexact FG-P-ADMM (Algorithm~\ref{alg:in:sy} with $\tau=0$) and  the inexact  G-P-ADMM (Algorithm~\ref{alg:in:sy} with $\theta=1$)   have been shown  very efficient  in some applications. 
 Indeed, 
as reported in \cite{adona2018partially},  the inexact FG-P-ADMM with    $\theta=1.6$  outperformed  other inexact ADMMs for two classes of problems, namely, LASSO and $\ell_1-$regularized logistic regression.  On the other hand,  the inexact G-P-ADMM, proposed later in  \cite{adonaCOAP}, with $\tau=0.9$ (or, $\alpha=1.9$ in term of the relaxation factor $\alpha$)  showed to be  even  more  efficient than  the  FG-P-ADMM with   $\theta=1.6$ for these same  classes of problems.  Therefore, our goal here is to investigate the efficiency  of Algorithm~\ref{alg:in:sy}, which  combines  both  acceleration parameters $\tau$ and $\theta$ in a single method, for solving another real-life application.
The computational results were obtained  using MATLAB R2018a on a  2.4 GHz Intel(R) Core i7 computer with 8 GB of RAM.

We use as test problem the total variation (TV) regularization problem (a.k.a. TV/L2 minimization), first proposed by \cite{Rudin1992259},
\begin{equation}\label{im:rest}
\min_{x\in \R^{m\times n}} \frac{\mu}{2}\Norm{Kx-c}^{2}+ \Norm{x}_{TV},
\end{equation}
where $x\in\R^{m\times n}$ is the original image to be restored,   $\mu$ is a positive  regularization parameter, $K:\R^{m\times n}\to \R^{m\times n}$ is a linear operator representing  some blurring operator, $c\in\R^{m\times n}$ is the  degraded image and $\norm{\cdot}_{TV}$ is the discrete TV-norm.  Let us briefly recall the definition of  TV-norm.
Let $x\in \R^{m\times n}$ be given and consider $D^1$ and $D^2$ the first-order finite difference $m\times n$ matrices in the horizontal and vertical directions, respectively, which, under the periodic boundary condition, are defined by
\begin{equation*}
(D^1x)_{i,j} = 
\begin{cases}
x_{i+1,j}- x_{i,j} &\text{if}\quad i< m,\\
x_{1,j}- x_{m,j} &\text{if}\quad i= m, 
\end{cases}
\qquad
(D^2x)_{i,j} = 
\begin{cases}
x_{i,j+1}- x_{i,j} &\text{if}\quad j<n,\\
x_{i,1}- x_{i,n} &\text{if}\quad j=n, 
\end{cases}
\end{equation*}
for $i=1,2,\ldots,m$ and $j=1,2,\ldots,n$. By defining  $D=\left(D^1;D^2\right)$, we obtain 
\begin{equation}\label{def:TV}
\Norm{x}_{TV}=\Norm{x}_{TV_{s}}:=\normiii*{Dx}_{s}:=\sum_{i=1}^{m}\sum_{j=1}^{n}\Norm{\left(Dx\right)_{i,j}}_{s},
\end{equation}
where $\left(Dx\right)_{i,j}=\left(\left(D^1x\right)_{i,j},\left(D^2x\right)_{i,j}\right)\in\R^{2}$ and $s=1$ or $2$. The  TV norm   is known as {anisotropic} and isotropic  if $s=1$ and  $s=2$,  respectively.  Here, we consider only the  isotropic case.

By introducing an auxiliary variable $y=(y^1,y^2)$ where $y^1,y^2\in\R^{m\times n}$ and, in view of the definition in \eqref{def:TV}, the problem in \eqref{im:rest} can be written as  
\begin{equation}\label{probl1}
\min_{x,y} \frac{\mu}{2}\Norm{Kx-c}^{2}+\normiii*{y}_{2}\quad s.t.\quad y = Dx,
\end{equation}
which is obviously an instance of \eqref{optl} with $f(x)=\frac{\mu}{2}\Norm{Kx-c}^{2}$, $g(y)=\normiii*{y}_{2}$, A$=-$D, B$=$I, and b$=0$.
In this case,  the pair $(\tilde{x}_k, u_k)$ in 
\eqref{cond:inex} can be obtained by computing  an approximate solution  $\tilde{x}_k$ with a residual $u_k$  of the following  linear system
\[
\left(\mu K^{\top}K+\beta D^{\top}D\right)x= \mu K^{\top}c + D^{\top}\left(\beta y_{k-1}-\gamma_{k-1}\right).
\]
In our implementation, the above linear system was reshaped  as a  linear system of size $mn\times1$ and then solved by means of the conjugate gradient method \cite{nocedal2006numerical} starting from the origin. 
Note that, by using the two-dimensional shrinkage operator \cite{Wang2008248,Yang2009569}, the subproblem  \eqref{g_sub} has a closed-form solution $y_k=\left(y^1_k,y^2_k\right)$ given explicitly by 
\[
\left(\left(y^1_k\right)_{i,j},\left(y^2_k\right)_{i,j}\right):=\max\left\{\Norm{(w^1_{i,j},w^2_{i,j})}-\frac{1}{\beta}, 0\right\}\left(\frac{w^1_{i,j}}{\Norm{(w^1_{i,j},w^2_{i,j})}},\frac{w^2_{i,j}}{\Norm{(w^1_{i,j},w^2_{i,j})}}\right),
\]
for $i=1,2,\ldots,m$ and $j=1,2,\ldots,n$, where 
\[\left(w^1,w^2\right):=(D^1\tilde{x}_{k}+(1/\beta)\gamma^1_{k-\frac{1}{2}}, D^2\tilde{x}_{k}+(1/\beta)\gamma^2_{k-\frac{1}{2}}),\] and the convention $0\cdot(0/0)=0$ is followed.

 The initialization parameters in  Algorithm~\ref{alg:in:sy} were set as follows: $(x_{0},y_{0},\gamma_{0})=(0,0,0)$, $\beta=1$, $G=I/\beta$,  $H={\bf 0}$ and $\hat{\sigma}=1-10^{-8}$. From \eqref{def:Reg} (see also Remark~\ref{remarkalg}(a)), for given $\tau\in\left(-1,1\right)$ and $\theta\in\left(-\tau,\left(1- \tau + \sqrt{5+2\tau-3\tau^2}\right)/2\right)$,  the  error tolerance parameter  $\tilde{\sigma}$ was defined as 
\begin{equation*}
\tilde{\sigma}=0.99\times
\begin{cases}
\min\left\{\dfrac{\left(1+\tau+\theta-\tau\theta-\tau^{2}-\theta^{2}\right)\left(\tau-1\right)}{\tau^{2}-2\theta+\theta^{2}}, 1-\tau, 1\right\}, &\!\!\text{if}\, \tau^{2}-2\theta+\theta^{2}<0,\\
\min\left\{1-\tau, 1\right\}, &\!\!\text{if}\, \tau^{2}-2\theta+\theta^{2}\geq 0.
\end{cases}
\end{equation*}
Moreover, we used the following stopping criterion 
\begin{equation*}\label{crit:stop}
\Norm{M(z_{k-1}-z_{k})}_{\infty}< 10^{-2},
\end{equation*}
where $z_{k}=(x_{k},y_{k},\gamma_{k})$ and $M$ is as in \eqref{def:oper}.

We considered six test images, which  were scaled in intensity to $[0,1]$, namely,  (a) Barbara  ($512\times 512$), (b) baboon  ($512\times 512$), (c)  cameraman ($256\times 256$), (d) Einstein ($225\times 225$), (e) clock $(256\times 256)$, and (f) moon $(347\times 403)$. 
All images were blurred by a 
Gaussian blur of size $9\times 9$   with standard deviation 5  and then corrupted  by a mean-zero Gaussian noise with variance $10^{-4}.$ The regularization parameter $\mu$ was set equal to  $10^3$. The quality of the images  was measured by the peak 
signal-to-noise ratio (PSNR) in decibel (dB):
\begin{equation*}
\text{PSNR} =10\log_{10}\left(\frac{\bar{x}_{\text{max}}^{2}}{\text{MSE}}\right)
\end{equation*}
where  $\text{MSE}=\frac{1}{mn}\sum_{i=1}^{m}\sum_{j=1}^{n}\left(\bar{x}_{i,j}-x_{i,j}\right)$, $\bar{x}_{\text{max}}$ is the maximum possible pixel value of the original image  and $\bar{x}$ and $x$ are the original image and the recovered image, respectively.

 Tables~\ref{tab:22}--\ref{tab:77}  report the numerical results of  Algorithm~\ref{alg:in:sy}, with some choices of $(\tau,\theta)$ 
satisfying \eqref{def:Reg}, for solving the six  TV regularization problem instances. In  the tables, ``Out" and ``Inner" denote the number of iterations and the total of inner iterations of the method, respectively, whereas ``Time" is the CPU time in seconds. We  mention that, for each problem instance, the final PSNRs  were the same for all $(\tau,\theta)$ considered. We displayed these values in  the tables as well as the PSNRs of the corrupted images.

\begin{table}[h!] 
\resizebox{\textwidth}{!}{
\begin{minipage}{0.5\textwidth}
\centering
\caption{Baboon $512\times 512$}\label{tab:22}
\begin{tabular}{cccrrr} 
\toprule
\multicolumn{6}{c}{PSNR: input 19.35dB, output 20.71dB}\\\midrule
$\tau$ &$\theta$ &$\tilde{\sigma}$ &Out &Inner &Time\\\midrule 
0.0 &1.00 &0.990 &131 &11723 &503.12\\  
0.0 &1.60 &0.062 &100 &11748 &495.38\\
0.9 &1.00 &0.099 &71  &7408  &314.51\\
0.7 &1.12 &0.175 &73  &7224  &312.27\\
0.7 &1.15 &0.142 &71  &7120  &303.42\\
0.7 &1.18 &0.107 &70  &7205  &309.97\\
0.8 &1.12 &0.074 &76  &8322  &387.54\\
0.8 &1.15 &0.040 &75  &8672  &393.05\\    
\bottomrule
\end{tabular}
\end{minipage}
\begin{minipage}{0.5\textwidth}
\centering
\caption{Barbara $512\times 512$}
\begin{tabular}{cccrrr}
\toprule
\multicolumn{6}{c}{PSNR: input 22.59dB, output 23.81dB}\\\midrule
$\tau$ &$\theta$ &$\tilde{\sigma}$ &Out &Inner &Time\\\midrule 
0.0 &1.00 &0.990 &142 &12910 &574.58\\ 
0.0 &1.60 &0.062 &105 &12403 &538.17\\
0.9 &1.00 &0.099 &80  &8620  &411.47\\
0.7 &1.12 &0.175 &84  &8643  &394.74\\
0.7 &1.15 &0.142 &82  &{8583}  &391.85\\ 
0.7 &1.18 &0.107 &82  &8835  &392.69\\
0.8 &1.12 &0.074 &{{79}}  &8665  &{371.47}\\
0.8 &1.15 &0.040 &{79}  &9110  &400.56\\ 
\bottomrule
\end{tabular} 
\end{minipage}  }
\end{table}

\begin{table}[h!]
\resizebox{\textwidth}{!}{
\begin{minipage}{0.5\textwidth}
\centering
\caption{Cameraman $256\times 256$}
\begin{tabular}{cccrrr} 
\toprule
\multicolumn{6}{c}{PSNR: input 21.02dB, output 25.14dB}\\\midrule
$\tau$ &$\theta$ &$\tilde{\sigma}$ &Out &Inner &Time\\\midrule 
0.0 &1.00 &0.990 &135 &13684 &87.92\\ 
0.0 &1.60 &0.062 &85  &10382 &64.21\\
0.9 &1.00 &0.099 &72  &8472  &54.83\\
0.7 &1.12 &0.175 &75  &8473  &52.77\\
0.7 &1.15 &0.142 &74  &8429  &52.02\\
0.7 &1.18 &0.107 &74  &8709  &53.41\\
0.8 &1.12 &0.074 &71  &8460  &51.83\\
0.8 &1.15 &0.040 &71  &8756  &53.75\\
\bottomrule
\end{tabular} 
\end{minipage}
\begin{minipage}{0.5\textwidth}
\centering
\caption{Clock $256\times 256$}
\begin{tabular}{cccrrr} 
\toprule
\multicolumn{6}{c}{PSNR: input 22.68dB, output 27.44dB}\\\midrule
$\tau$ &$\theta$ &$\tilde{\sigma}$ &Out &Inner &Time\\\midrule 
0.0 &1.00  &0.990 &130 &12666 &82.03\\  
0.0 &1.60  &0.062 &84  &10104 &64.97\\
0.9 &1.00  &0.099 &69  &7746 &53.39\\
0.7 &1.12  &0.175 &72  &7807 &51.95\\
0.7 &1.15  &0.142 &73  &7985 &52.01\\
0.7 &1.18  &0.107 &70  &7767 &55.67\\
0.8 &1.12  &0.074 &68  &7740 &50.03\\
0.8 &1.15  &0.040 &67 &7994 &51.48\\
\bottomrule
\end{tabular} 
\end{minipage} }
\end{table}

\begin{table}[h!]
\resizebox{\textwidth}{!}{
\begin{minipage}{0.5\textwidth}
\centering
\caption{Einstein $225\times 225$}
\begin{tabular}{cccrrr} 
\toprule
\multicolumn{6}{c}{PSNR: input 23.70dB, output 28.24dB}\\\midrule
$\tau$ &$\theta$ &$\tilde{\sigma}$ &Out &Inner &Time\\\midrule 
0.0 &1.00  &0.990 &120 &10506 &56.86\\ 
0.0 &1.60  &0.062 &88  &9968  &52.84\\ 
0.9 &1.00  &0.099 &72  &7560  &40.47\\ 
0.7 &1.12  &0.175 &68  &6573  &34.54\\
0.7 &1.15  &0.142 &74  &7631  &40.88\\
0.7 &1.18  &0.107 &73  &7566  &40.69\\
0.8 &1.12  &0.074 &71  &7685  &40.00\\
0.8 &1.15  &0.040 &70  &7830  &40.27\\
\bottomrule
\end{tabular}
\end{minipage}
\begin{minipage}{0.5\textwidth}
\centering
\caption{Moon $347\times 403$}\label{tab:77} 
\begin{tabular}{cccrrr} 
\toprule
\multicolumn{6}{c}{PSNR: input 25.57dB, output 28.28dB}\\\midrule
$\tau$ &$\theta$ &$\tilde{\sigma}$ &Out &Inner &Time\\\midrule 
0.0 &1.00  &0.990 &128 &11684 &249.27\\
0.0 &1.60  &0.062 &88  &10239 &215.91\\
0.9 &1.00  &0.099 &72  &7828  &168.30\\
0.7 &1.12  &0.175 &76  &7921  &170.00\\
0.7 &1.15  &0.142 &74  &7796  &205.34\\
0.7 &1.18  &0.107 &73  &7909  &194.64\\
0.8 &1.12  &0.074 &68  &7412 &161.05\\
0.8 &1.15  &0.040 &67  &7711  &181.89\\
\bottomrule
\end{tabular}
\end{minipage} }
\end{table}

 From the tables, we can see clearly  the numerical benefits of using  acceleration parameters  $\tau >0$ and $\theta>1$. Note  that Algorithm~\ref{alg:in:sy} with the  choice  $(\tau,\theta)=(0,1)$  had the worst performance, in terms of the three performance measurements, for all problem instances.  Note also that Algorithm~\ref{alg:in:sy} with $(\tau,\theta)=(0.9,1)$ ($0.9$ was the best value for $\tau$ in \cite{adonaCOAP}) performed better than  Algorithm~\ref{alg:in:sy} with $(\tau,\theta)=(0,1.6)$ ($1.6$ was the best value for $\theta$ in \cite{adona2018partially}), such  behavior was also observed  in \cite{adonaCOAP} for the LASSO and $\ell_1-$regularized logistic regression problems. We stress that  Algorithm~\ref{alg:in:sy} with  $(\tau,\theta)=(0.8,1.12)$ was faster in four (Barbara, cameraman, clock and moon) of six instances.  Fig.~\ref{figCamer} plots the original and corrupted  images as well as the restored image by    Algorithm~\ref{alg:in:sy} with $(\tau,\theta)=(0.8,1.12)$ for the six instances.
As a summary, we can conclude that  combinations of the acceleration parameters $\tau$ and $\theta$  can also be  efficient strategies in the inexact ADMMs for solving real-life applications.

\begin{figure}[!h]
\vspace{-0.5cm}
\includegraphics[width=\textwidth, height=3.7cm]{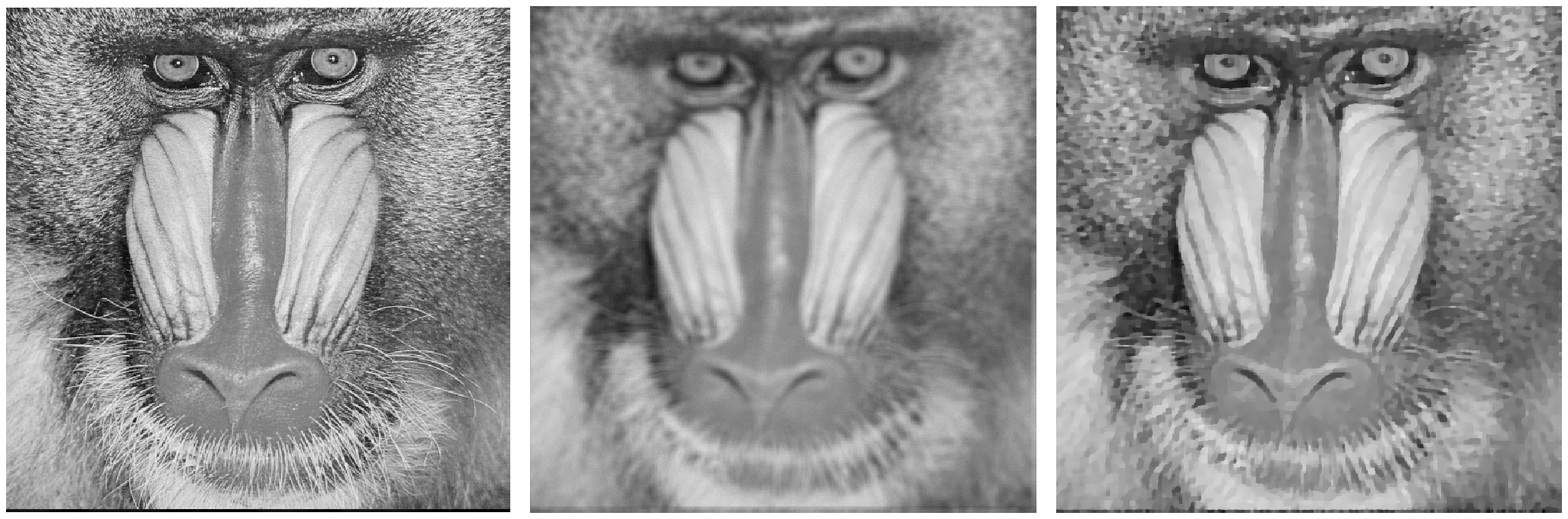}\vspace{-0.5cm}
\includegraphics[width=\textwidth, height=3.7cm]{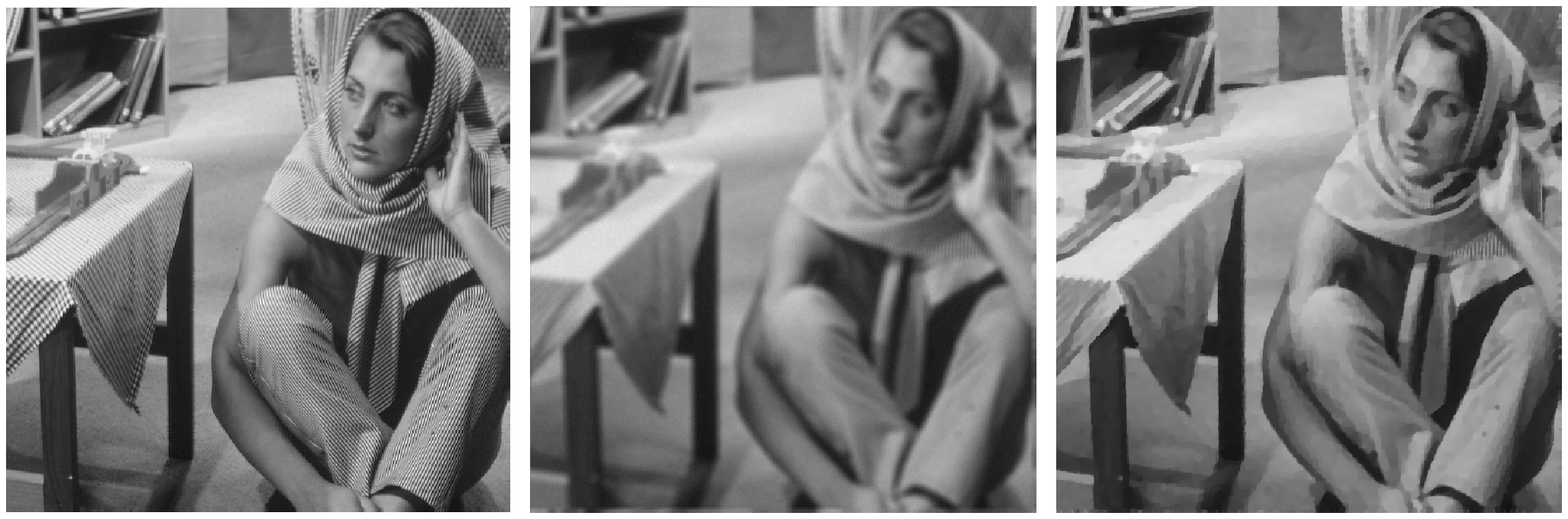}\vspace{-0.5cm}
\includegraphics[width=\textwidth, height=3.7cm]{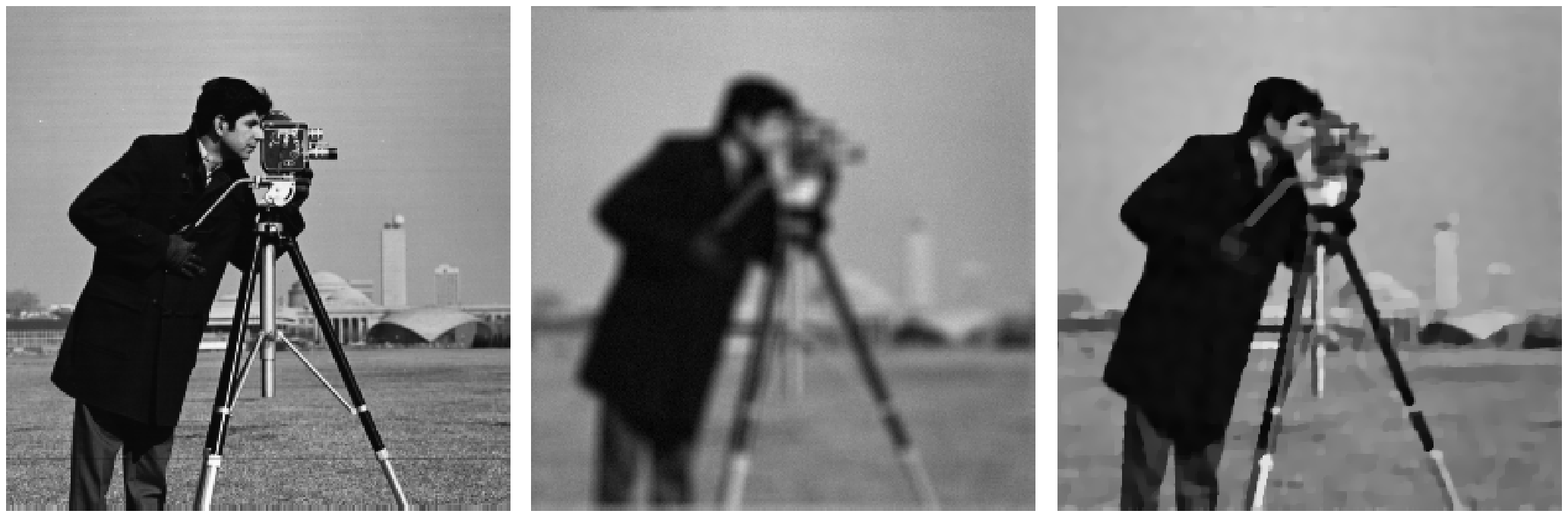}\vspace{-0.5cm}
\includegraphics[width=\textwidth, height=3.7cm]{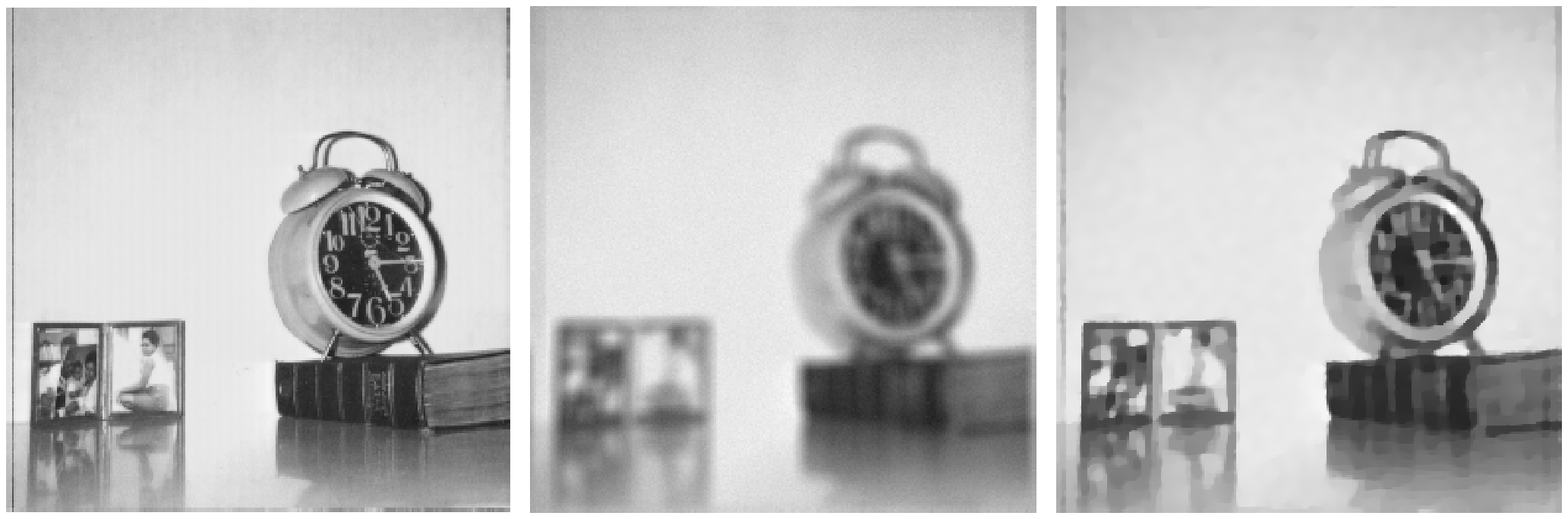}\vspace{-0.5cm}
\includegraphics[width=\textwidth, height=3.7cm]{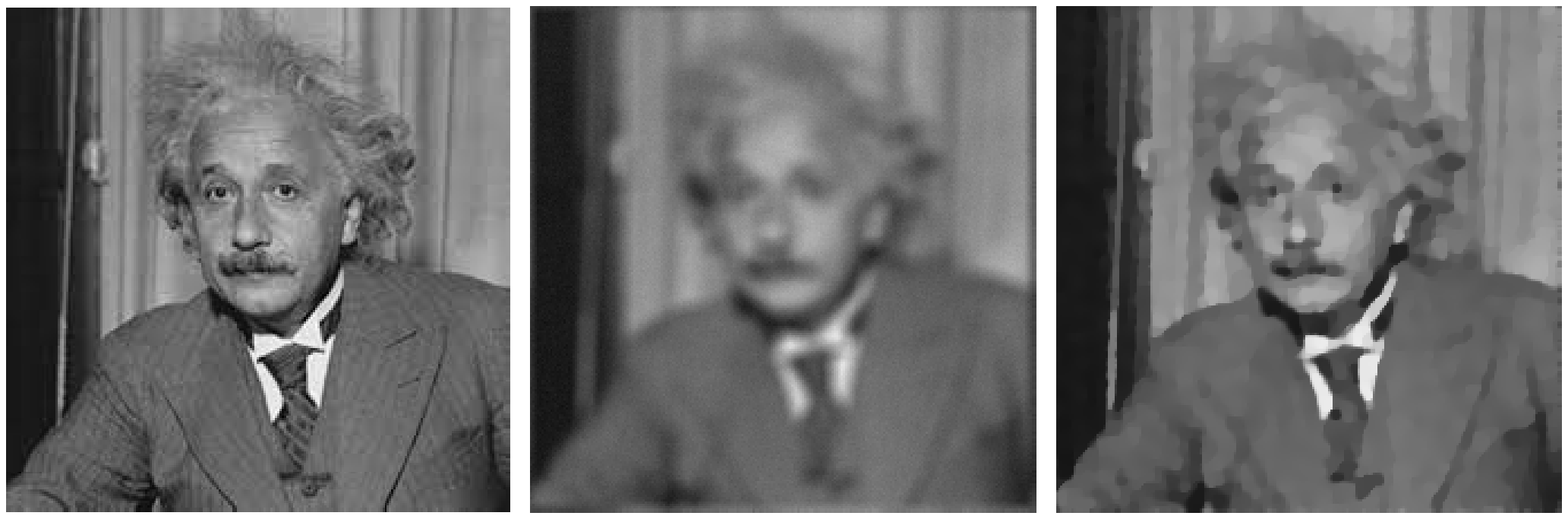}\vspace{-0.5cm}
\includegraphics[width=\textwidth, height=3.7cm]{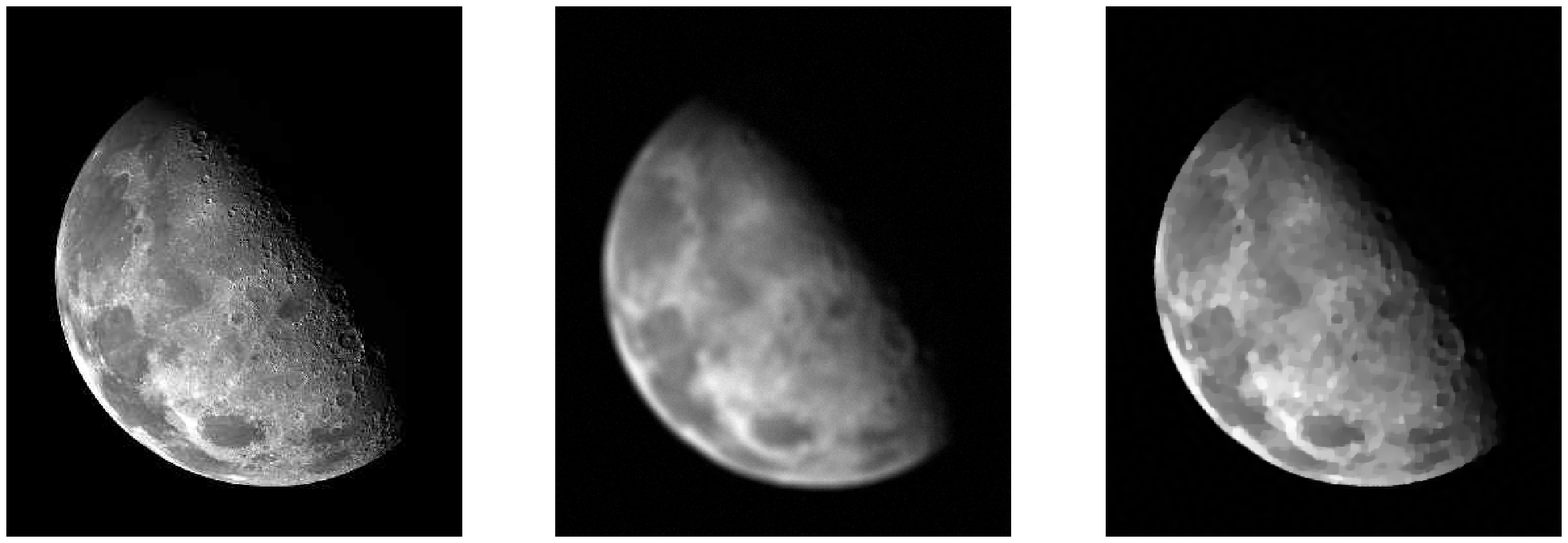}\vspace{-0.6cm}
\caption{Results on the images (top to bottom): ``Baboon", ``Barbara", ``Cameraman", ``Clock", ``Einstein"  and ``Moon". First column is the original images, the second is blurred and noisy images,  and the third is the restored images by Algorithm~\ref{alg:in:sy} with $(\tau,\theta)=(0.8,1.12)$.} \label{figCamer}
\end{figure}

\section{Final remarks}\label{conclusion}

We proposed   an  inexact  symmetric proximal  ADMM for solving  linearly constrained optimization problems. Under appropriate hypotheses,   the global  $\mathcal{O} (1/ \sqrt{k})$  pointwise  and $\mathcal{O} (1/ {k})$ ergodic  convergence rates of the proposed method were established   for a  domain of the acceleration parameters, 
 which is consistent  with the largest known one in the exact case.
Numerical experiments were carried out in order to illustrate the numerical behavior of the new method. They indicate that the proposed scheme represents an useful tool for solving   real-life applications. To the best of our knowledge, this was the first time that an  inexact variant of the  symmetric proximal ADMM was proposed and analyzed.


\def\cprime{$'$}

\end{document}